\documentclass[1p]{elsarticle}
\usepackage{amsmath}
\usepackage{hyperref}
\usepackage{amssymb}
\usepackage{latexsym}
\usepackage{amscd}
\usepackage{graphicx} 
\usepackage{amsthm}
\usepackage{mathrsfs}
\usepackage{xypic}
\usepackage{bm}
\usepackage{color}
\usepackage{ulem}

\newdimen\AAdi%
\newbox\AAbo%
\def\AAk#1#2{\s_etbox\AAbo=\hbox{#2}\AAdi=\wd\AAbo\kern#1\AAdi{}}%
\def\AAr#1#2#3{\s_etbox\AAbo=\hbox{#2}\AAdi=\ht\AAbo\raise#1\AAdi\hbox{#3}}%
\font\tenmsb=msbm10 at 12pt \font\sevenmsb=msbm7 at 8pt
\font\fivemsb=msbm5 at 6pt
\newfam\msbfam
\textfont\msbfam=\tenmsb \scriptfont\msbfam=\sevenmsb
\scriptscriptfont\msbfam=\fivemsb
\def\Bbb#1{{\tenmsb\fam\msbfam#1}}
\newtheorem{thm}{Theorem}[section]
\newtheorem{lem}[thm]{Lemma}
\newtheorem{cor}[thm]{Corollary}

\newtheorem{pro}[thm]{Proposition}
\newtheorem{defi}[thm]{Definition}

\newcommand{\Section}[2]{\setcounter{equation}{0}
\allowdisplaybreaks
\section[#1]{#2}}

\def\pr {\noindent {\it Proof.} }

\def\n{\nabla}

\def\f#1#2{\frac{#1}{#2}}

\def\mc#1{\mathcal{#1}}
\def\pr{\frac {\partial}{\partial r}}

\def\pf#1{\frac{\partial}{\partial #1}}
\def\pd#1#2{\frac {\partial #1}{\partial #2}}

\def\td{\tilde}

\def\a{\alpha}

\def\p#1{\partial #1}

\def\de{\delta}
\def\De{\Delta}

\def\ep{\varepsilon}

\def\g{\gamma}
\def\k{\kappa}
\def\la{\lambda}
\def\La{\Lambda}
\def\om{\omega}
\def\Om{\Omega}
\def\th{\theta}
\def\Th{\Theta}

\def\w{\wedge}

\def\Hess{\mbox{Hess}}
\def\R{\Bbb{R}}

\def\lan{\langle}
\def\ran{\rangle}
\def\ra{\rightarrow}

\def\mb{\mathbf}

\def\Hess{\text{Hess }}

\begin{document}
\title
{Dirichlet boundary values on Euclidean balls with infinitely many solutions for the minimal surface system}

  \author[x1,x2]{Xiaowei Xu}
 \ead{xwxu09@ustc.edu.cn}
 \author[y1,y2]{Ling Yang}
 \ead{yanglingfd@fudan.edu.cn}
 \author[z1,z2]{Yongsheng Zhang}
 \ead{yongsheng.chang@gmail.com}

\address[x1]{School of Mathematical Sciences,  University of Science and Technology of China, Hefei, 230026, Anhui province, China}
\address[x2]{Wu Wen-Tsun Key Laboratory of Mathematics, USTC, Chinese Academy of Sciences, Hefei, 230026, Anhui province, China}
\address[y1]{School of Mathematical Sciences, Fudan University, Shanghai, 200433, China}
\address[y2]{Shanghai Center for Mathematical Sciences, Fudan University, Shanghai, 200438, China}
\address[z1]{School of Mathematical Sciences, Tongji University, Shanghai, 200092, China}
\address[z2]{Institute for Advanced Study, Tongji University, Shanghai, 200092, China}






\begin{abstract}
We make systematic developments on Lawson-Osserman constructions relating to the Dirichlet problem (over unit disks) for minimal surfaces of high codimension in their 1977' Acta paper. In particular, we show the existence of boundary functions for which infinitely many analytic solutions and at least one nonsmooth Lipschitz solution exist simultaneously. This newly-discovered amusing phenomenon enriches the understanding on the Lawson-Osserman philosophy.
\\{\ }\\
\textbf{R\'esum\'e}
\\{\ }\\
On donne des d\'eveloppements syst\'emetiques sur les constructions \`a la Lawson-Osserman relatives au probl\`eme de Dirichlet (sur les disques unitaires) pour des surfaces minimales de codimension \'elev\'ee dans leur papier Acta en 1977. En particulier, nous d\'emontrons l'existence des fonctions aux limites pour lesquelles une infinit\'e de solutions analytiques et au moins une solution de Lipschitz (non lisse) existent simultan\'ement. Cette nouvelle et amusante d\'ecouverte enrichit notre compr\'ehension de la philosophie \`a la Lawson-Osserman.
\\{\ }
\end{abstract}


\maketitle

\tableofcontents

\renewcommand{\proofname}{\it Proof.}


\Section{Introduction}{Introduction}

The research on minimal graphs in Euclidean spaces has a long and fertile history.
Among others, the \textbf{Dirichlet problem} (cf. \cite{j-s,b-d-m, de, m1,l-o}) is a central topic in this subject:

\textit{Let $\Omega\subset \R^{d_1}$ be a bounded and strictly convex domain with boundary of class $C^r$ for $r\geq 2$.
        It asks, for a given function $f:\partial \Omega \rightarrow \mathbb R^{d_2}$ of class $C^s$ with $0\leq s\leq r$,
         what kind of and how many functions $\in C^0(\overline{\Omega};\R^{d_2})\bigcap \text{Lip}(\Omega;\R^{d_2})$
          exist
        so that each such function $F:x=(x^1,\cdots,x^{d_1})\mapsto F(x)=(F^1,\cdots,F^{d_2})$ is a weak solution to the minimal surface system
         \begin{equation}
         \left\{\begin{array}{cc}
         \sum\limits_{i=1}^{d_1}\pf{x^i}(\sqrt{g}g^{ij})=0, & j=1,\cdots,d_1,\\
         \sum\limits_{i,j=1}^{d_1}\pf{x^i}(\sqrt{g}g^{ij}\pd{F^\a}{x^j})=0, & \alpha=1,\cdots,d_2,
         \end{array}
         \right.
         \end{equation}
 where $g_{ij}=\de_{ij}+\sum\limits_{\a=1}^{d_2}\pd{F^\a}{x^i}\pd{F^\a}{x^j}$, $(g^{ij})=(g_{ij})^{-1}$ and $g=\det(g_{ij})$,
 satisfying the Dirichlet condition
         \[
          F|_{\partial \Omega}=f.
          \]
        That means
        the graph of $F$ is minimal in the sense of \cite{al} with that of $f$ being its boundary.
        }

When $d_2=1$, we have a fairly profound understanding.

\begin{itemize}
\item Given arbitrary boundary data of class $C^0$,
                    by the works of J. Douglas \cite{d},
                    T. Rad\'o \cite{r,r2},
                    Jenkins-Serrin \cite{j-s}
                    and Bombieri-de Giorgi-Miranda \cite{b-d-m},
                    there exists a unique Lipschitz solution to the Dirichlet problem.
\item Furthermore, due to the works of E. de Giorgi \cite{de} and J. Moser \cite{m1}, this solution turns out to be analytic.
\item Each solution gives an absolutely area-minimizing graph
          by virtue of the convexity of $\Omega\times \mathbb R$ and \S5.4.18 of
          \cite{fe}.
          As a consequence, it is stable.
\end{itemize}

However, the situation for $d_2\geq 2$ becomes much more complicated.
          Even when $\Om=\mathbb D^{d_1}$ (the unit Euclidean disk),
          H. B. Lawson and R. Osserman \cite{l-o}
          discovered astonishing phenomena that reveal essential differences.

\begin{itemize}
          \item For $d_1=2$, $d_2\geq 2$, some real analytic boundary data can be constructed
                   so that
                   there exist at least three different analytic solutions to the Dirichlet problem.
                   Moreover,
                   one of them gives an unstable minimal surface. (See \cite{sa} for more discussions.)
          \item For $d_1\geq 4$ and $d_1-1\geq d_2\geq 3$, the Dirichlet problem is generally not solvable.
                   In fact, for each $f:S^{d_1-1}\ra S^{d_2-1}$ that is not homotopic to zero,
                   there exists a positive constant $c$ depending only on $f$, such that
                   the problem is unsolvable for the boundary data $f_\varphi:=\varphi\cdot f$, where
                   $\varphi$ is a constant no less than $c$.
          \item For certain boundary data, there exists a Lipschitz solution to the Dirichlet problem which is not $C^1$.
\end{itemize}

As mentioned in \cite{l-o} the nonexistence and irregularity of the Dirichlet problem are intimately related as follows.
Given $f$ that represents a non-trivial element of $\pi_{d_1-1}(S^{d_2-1})$, the Dirichlet problem for $f_\varphi$ is solvable when $\varphi$ is small (due to the implicit function theorem)
but unsolvable for large $\varphi$. This leads Lawson-Osserman to the
     \textbf{philosophy} that there should exist a critical value $\varphi_0$
which supports some sort of singular solution.
In particular, for Hopf map $H^{2m-1,m}:S^{2m-1}\rightarrow S^m$ with $m=2,4$ or $8$,
\begin{equation}
M_m:=\{ (\cos\th_m\cdot x,\sin\th_m\cdot H^{2m-1,m}(x)):x\in S^{2m-1}\}\subset S^{3m}
\end{equation}
with
\begin{equation}
\th_m:=\arccos\sqrt{\f{4(m-1)}{3(2m-1)}}
\end{equation}
gives a minimal sphere
and spans a minimal cone $C_m$ which is exactly the graph of 
\begin{equation}F_{m}(y)=\left\{\begin{array}{cc}
\tan\th_m\cdot |y|\cdot H^{2m-1,m}(\f{y}{|y|}) & y\neq 0,\\
0 & y=0.
\end{array}\right.
\end{equation}
The restriction of the graph over domain $\Bbb{D}^{2m}$ presents a Lipschitz solution to the Dirichlet problem for boundary data $\tan\th_m\cdot H^{2m-1,m}$.

To further explore the topic in a more general framework,
we introduce the following concepts.

\begin{defi}
For a smooth map $f:S^n\ra S^m$,
if there exists an acute angle $\th$, such that
\begin{equation}
M_{f,\th}:=\{(\cos\th\cdot x,\sin\th\cdot f(x)):x\in S^n\}
\end{equation}
is a minimal submanifold of $S^{n+m+1}$, then we call $f$ a {\bf Lawson-Osserman map (LOM)},
$M_{f,\th}$ the associated {\bf Lawson-Osserman sphere (LOS)}, and the cone $C_{f,\th}$ over $M_{f,\th}$ the corresponding {\bf Lawson-Osserman cone (LOC)}.
\end{defi}

Similarly, for an LOM $f$, the associated $C_{f,\th}$ is the graph of
\begin{equation}\label{cone-graph}
F_{f,\th}(y)=\left\{\begin{array}{cc}
\tan\th \cdot |y|\cdot f(\f{y}{|y|}) & y\neq 0,\\
0 & y=0.
\end{array}\right.
\end{equation}
Thus the restriction over $\Bbb{D}^{n+1}$ provides a Lipschitz solution to the Dirichlet problem
for boundary data $ f_{\varphi_0}$
where $\varphi_0=\tan\th$.

 Assume $f:S^n\ra S^m$ is an LOM, not totally geodesic.
 Then $f$ is called an {\bf LOMSE} if
 all nonzero singular values
of $(f_*)_x$ are equal for each $x\in S^n$. Here, a singular value of $(f_*)_x$ means a critical value of the
 function $v\in T_x S^n\mapsto \frac{|f_* v|}{|v|}$. As $x$ varies, these values form a continuous function $\la(x)$.
It can be shown that $\la(x)$ equals
 a constant $\la$ and that $f$ has constant rank $p$ (see Theorem \ref{g2} (ii)). Moreover, all components of this vector-valued function $f$, i.e. $f_1,\cdots,f_{m+1}$ are
  spherical harmonic functions of degree $k\geq 2$ (see Theorem \ref{npk}).
  Accordingly, such $f$ is called an {\bf LOMSE of (n,p,k)-type}.

In this paper we shall systematically study LOMSEs from several viewpoints.

A characterization of LOMSEs will be established in Theorem \ref{g2}, which asserts that
each of them can be written as the composition of a Riemannian submersion from $S^n$ with connected fibers and (up to a scaling) an isometric minimal immersion into $S^m$.
In fact,
by the wonderful results in \cite{br,g-g,w},
the submersion
that determines $(n,p)$
has to be a Hopf fibration over a complex projective space, a quaterninonic projective space
or the octonionic projective line;
 while the choices of the second component
 for each even integer $k$ usually form
a moduli space of large dimension (see \cite{c-w,oh,u,to,to2}), yielding a huge number of LOMSEs as well as the associated LOSs and LOCs.
Note that, except the original three
Lawson-Osserman cones,
we always have $m>n$
that means `$f$ is not homotopic to zero' is not a
                            requisite to span a non-parametric minimal cone.

Although we find uncountably many LOMSEs,
each of them has the nonzero singular value $\la$ and the acute angle $\th$
both in a discrete manner in terms of $(n,p,k)$ (see Theorem \ref{npk}).
Consequently, we observe interesting gap phenomena for certain geometric quantities of LOSs or LOCs associated to LOMSEs, e.g. angles between normal planes and a fixed reference plane, volumes, Jordan angles and slope functions (see Corollary \ref{cor2}).
Rigidity properties for
these quantities of compact minimal submanifolds in spheres or entire minimal graphs in Euclidean spaces have drawn attention in
many literatures \cite{b,fc, c-l-y, j-x,j-x-y3,j-x-y2,j-x-y4,j-x-y5}.

We seek for analytic solutions to Dirichlet problem
for the boundary data $f_\varphi$ 
as well.
 A good candidate (compared with \eqref{cone-graph}) turns out to be
 \begin{equation}
 F_{f,\rho}(y)=\left\{\begin{array}{cc}
\rho(|y|)f(\f{y}{|y|}) & y\neq 0\\
0 & y=0
\end{array}\right.
\end{equation}
Here $\rho$ is a smooth positive function on $(0,b)$ for some $b\in\R_+\cup \{+\infty\}$,
satisfying $\lim\limits_{r\ra 0^+}\rho=0$.
If
\begin{equation}\label{frho1}
M_{f,\rho}:=\big\{(rx,\rho(r)f(x)):x\in S^n, r\in (0,b)\big\}
\end{equation}
 is minimal and $\rho_r(0)=0$,
 then Morrey's regularity theorem \cite{mo}
 ensures $F_{f,\rho}$ an analytic solution to the minimal surface equations through the origin.
 Since the minimality
 is invariant under rescaling,
 $F_{f,\rho_d}$
 for
  $\rho_d(r):=\f{1}{d}\,\rho(d\cdot r)$
  and $d>0$
  produce a series of minimal graphs.
 So, in the $r\rho$-plane,
 every intersection point of the graph of $\rho$ and the ray $\rho=\varphi\cdot r$
 generates an analytic solution to the Dirichlet problem for $f_\varphi$.

 In particular, when $f$ is an LOMSE, the minimal surface equations can be reduced to (\ref{ODE1}),
 a nonlinear ordinary differential equation of second order,
 equivalent to an autonomous system (\ref{ODE2}) in the $\varphi\psi$-plane
 for $\varphi:=\f{\rho}{r}$, $t:=\log r$ and $\psi:=\varphi_t$.
 With the aid of suitable barrier functions, we obtain a long-time existing bounded solution, whose orbit in the phase space
 emits from the origin - a saddle critical point and limits to $P_1(\varphi_0,0)$ - a stable critical point (see Propositions \ref{case1}-\ref{case2}).

Quite subtly, there are two dramatically different types of asymptotic behaviors
aroud $P_1$ relying on the values $(n,p,k)$ of $f$:
\begin{enumerate}
\item [(I)]
$P_1$ is a stable center when $(n,p,k)=(3,2,2), (5,4,2), (5,4,4)$ or $n\geq 7$;
\item [(II)] $P_1$ is a stable spiral point when $(n,p)=(3,2)$,
$k\geq 4$ or $(n,p)=(5,4)$, $k\geq 6$.
\end{enumerate}
                   \begin{figure}[h]
                              \begin{minipage}[c]{0.5\textwidth}
                       \   \ \    \includegraphics[scale=0.55]{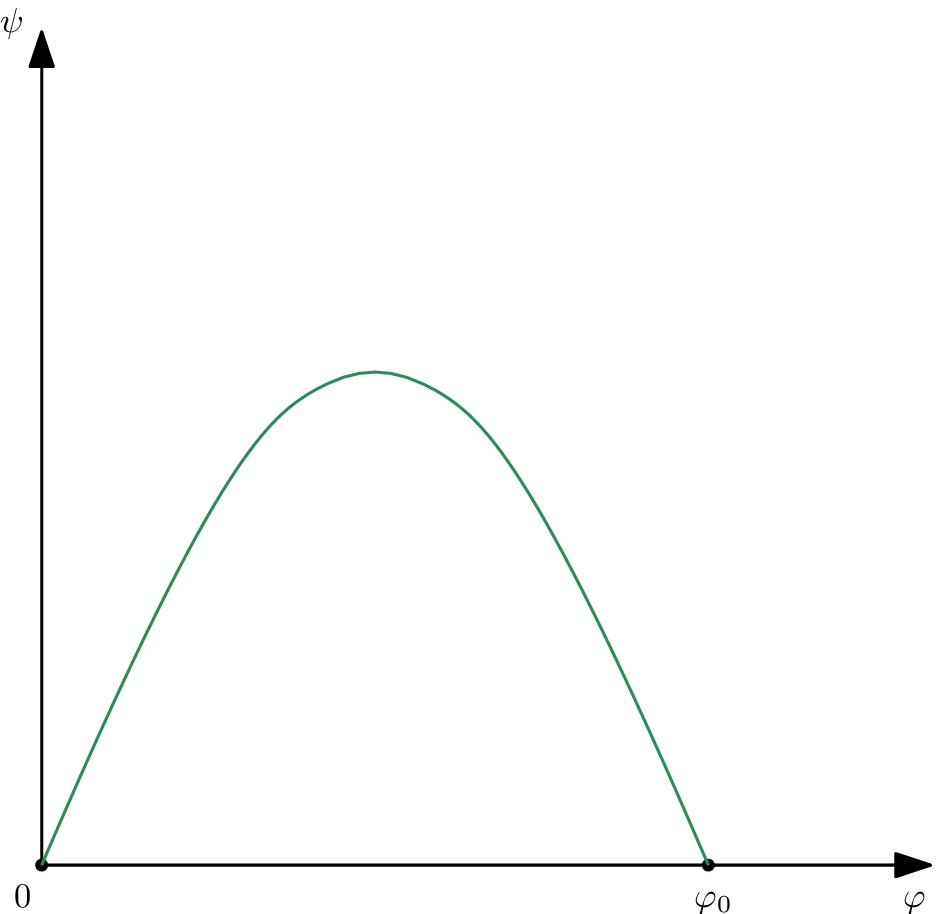}
                              \end{minipage}%
                          \begin{minipage}[c]{0.45\textwidth}
                           \includegraphics[scale=0.5]{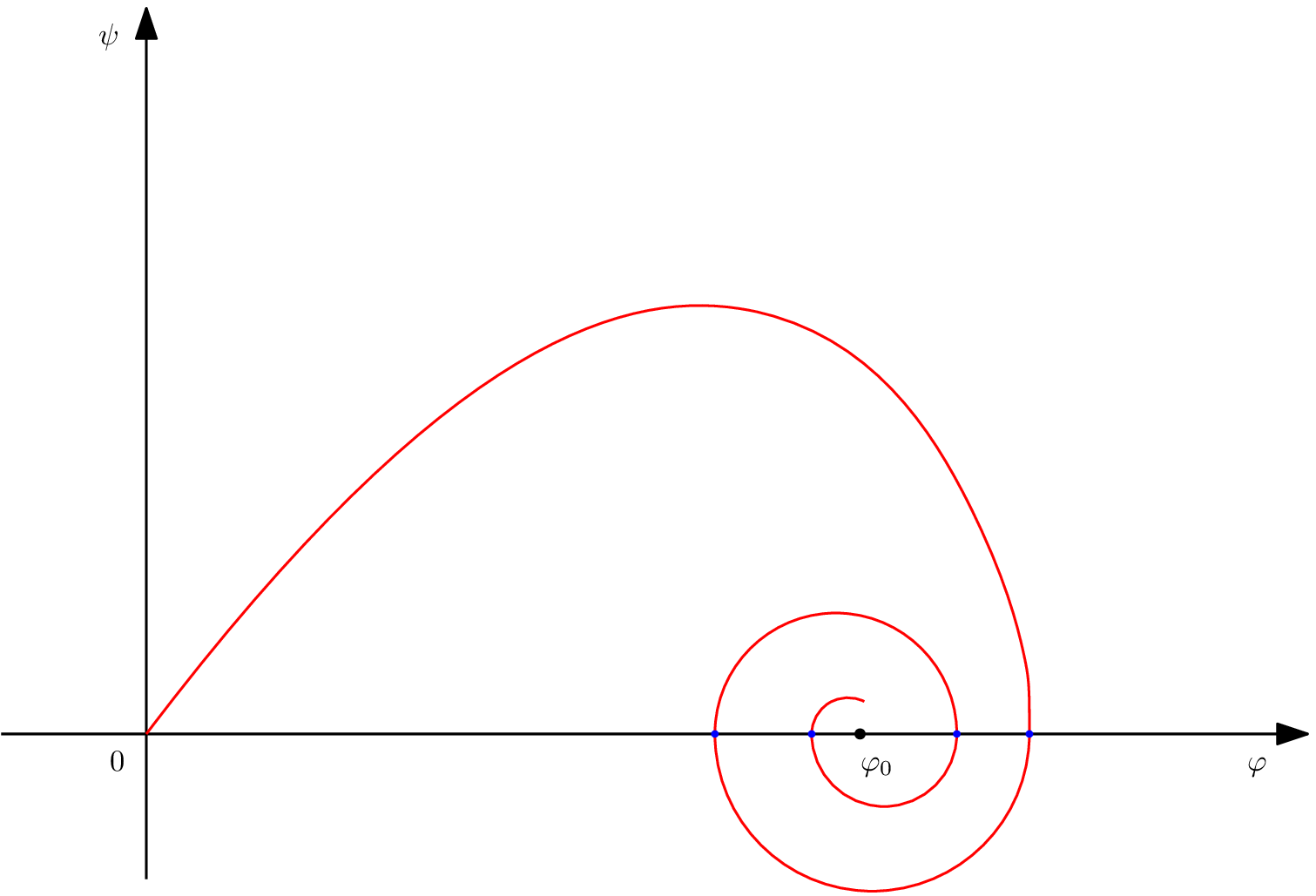}
                           \end{minipage}
                    \end{figure}
 Corresponding graphs of the solutions $\rho$ to (\ref{ODE1}) are illustrated below, respectively. 
                              $$\begin{minipage}[c]{0.55\textwidth}
                              \includegraphics[scale=0.45]{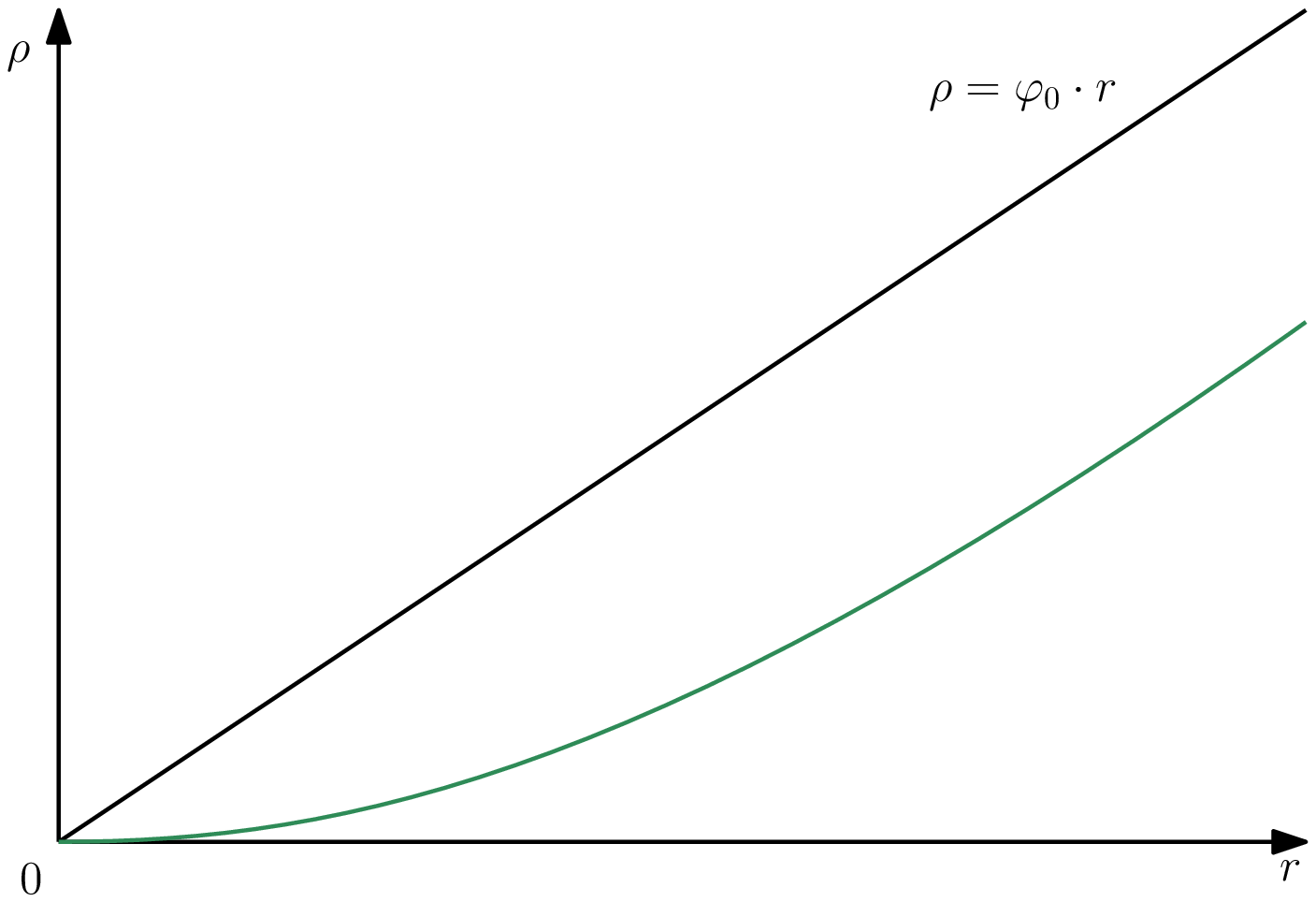}
                              \end{minipage}%
                          \begin{minipage}[c]{0.55\textwidth}
                           \includegraphics[scale=0.45]{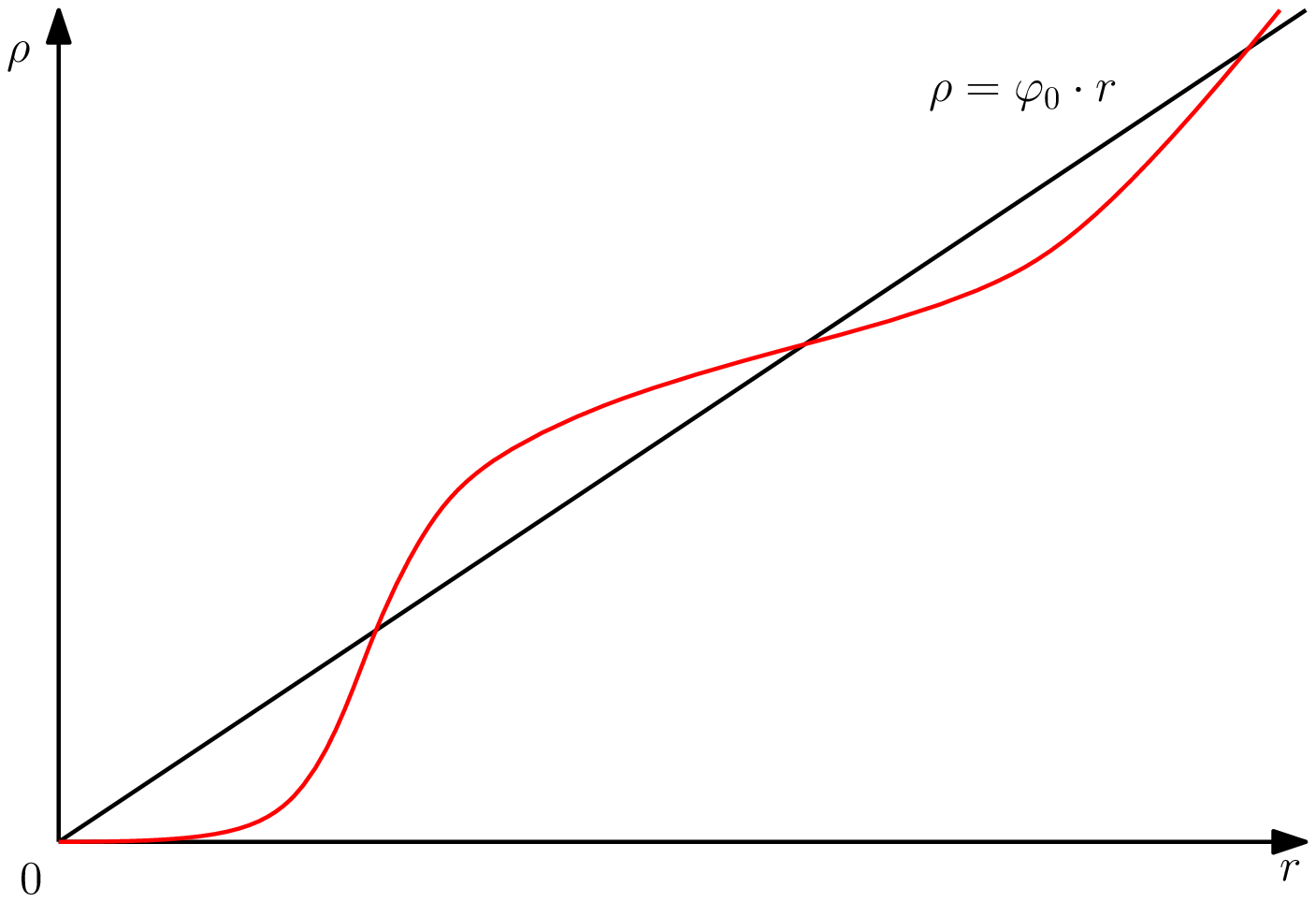}
                           \end{minipage}$$

 Much interesting information can be read off from the above pictures:
\begin{enumerate}
                  \item [(A)]
                              For each LOMSE $f$, there exists an entire analytic minimal graph whose tangent cone at infinity is exactly the LOC
                               associated to $f$ (see Theorem \ref{graph}).
                  \item [(B)]
                              For an LOMSE $f$ of Type (II),
                              there exist infinitely many analytic solutions to the Dirichlet problem
                              for $f_{\varphi_0}$; meanwhile, it also has a singular Lipschitz solution
                              which corresponds to the truncated LOC (see Theorem \ref{uni}).
                  \item[(C)] For Type (II),
                              although a Lipschitz solution arises for the boundary data $f_{\varphi_0}$,
                              there exists an $\epsilon>0$ such that
                                     the Dirichlet problem
                                      still has
                                     analytic solutions for $f_\varphi$
                                     whenenver
                                     $\varphi\in (\varphi_0,\varphi_0+\epsilon)$.
                  \item [(D)]
                              By the monotonicity of density for minimal submanifolds (currents) in Euclidean spaces (see \cite{fe,c-m}), LOCs associated to LOMSEs of Type (II) are all non-minimizing (see Theorem \ref{non-min}).

                  \end{enumerate}

 To the authors' knowledge, it seems to be the first time to have phenomena (B)-(C) observed,
                  and hard to foresee
                  the occurrence from the classical theory of partial differential equations.
                  Note that
                  the non-uniqueness (``at least three") 
                  in Lawson-Osserman \cite{l-o} heavily relies on
                  dimension $n=1$ to apply the work of T. Rad\'o. 
                  In our cases, $n$ can be $3$ or $5$ and
                  these examples perfectly demonstrate the non-uniquenss
                  of infinitely many, and moreover, the smooth and nonsmooth simultaneously.
                  We believe more mysteries and beauties hide behind. 

By the machinery of calibrations, the LOC associated to the Hopf map from $S^3$ onto $S^2$
(i.e. the LOMSE of $(3,2,2)$-type)
was shown area-minimizing by Harvey-Lawson \cite{h-l}.
It would be interesting to consider
 whether the associated LOC is area-minimizing for an LOMSE of $(n,p,k)$-type.
 In Theorem \ref{non-min}, we establish a partial negative answer to the question. On the other hand,
 in a subsequent paper \cite{x-y-z}, we explore this subject from a different point of view and
 confirm that all LOCs associated to LOMSEs of $(n,p,2)$-type are
area-minimizing (useful in geometric measure theory, e.g. see \cite{fe, z}).

\bigskip\bigskip

\Section{Lawson-Osserman maps}{Lawson-Osserman maps}

\subsection{Preliminaries on harmonic maps}
Let $(M^n,g)$ and $(N^m,h)$ be Riemannian manifolds
 and $\phi$ be a smooth mapping from $M$ to $N$.
The \textit{energy desity} of $\phi$
at $x\in M$
is defined to be
\begin{equation}
e(\phi):=\f{1}{2}\sum_{i=1}^n h(\phi_* e_i,\phi_* e_i).
\end{equation}
Here $\{e_1,\cdots,e_n\}$ is an orthonormal basis of $T_xM$.
The \textit{total energy} $E(\phi)$ is
the integral of $e(\phi)$ over $M$.

Let $\td{\n}$ be  Levi-Civita connection for $g$ and $\n$ the connection on $\phi^{-1}TN$ compatible with $h$.
Then the \textit{second fundamental form} of $\phi$ is given by
\begin{equation}
B_{XY}(\phi):=\n_{ X}(\phi_* Y)-\phi_*(\td{\n}_X Y),
\end{equation}
whose trace under $g$ is the \textit{tensor field} of $\phi$
\begin{equation}
\tau(\phi):=\sum_{i=1}^n B_{e_ie_i}(\phi).
\end{equation}
 If $\tau(\phi)$ vanishes indentically,
 then $\phi$ is called a \textit{harmonic map}.
 When $B\equiv 0$,
 $\phi$ is called \textit{totally geodesic}.
 It is well known that
 $\phi$ is harmonic if and only if it is a critical point of
 functional $E$.

For a smooth (vector-valued) function $f:(M,g)\ra \R^n$,
it follows $\tau(f)=\De_g(f)$
where $\De_g$ is the {\it Laplace-Beltrami operator} for $g$.
Hence the harmonicity is the same as that in the usual sense.

Given an isometric immersion $i: (M,g)\ra (N,h)$,
its second fundamental form can be
identified with the second fundamental form of $M$ in $N$
and
its tensor field regarded as the mean curvature vector field $\mathbf{H}$.
Therefore, $i$ is  harmonic if and only if it is an isometric minimal immersion,
and 
totally geodesic if and only if it is an isometric totally geodesic immersion.

In the case of Riemannian submersions, a classical result is the following.

 \begin{pro}\label{ER} (see e.g. Proposition 1.12 of \cite{e-r})
 A Riemannian submersion $\pi:(M,g)\ra (N,h)$ is harmonic if and only if each fiber of $\pi$ is minimal. 
 \end{pro}

For Riemannian manifolds $M^n, N^m, \bar{N}$
 and
 smooth maps $\phi: M^n\ra N^m$, $\bar{\phi}:N^m \ra \bar{N}$,
a fundamental composition formula for tension fields (see Proposition 1.14 in \cite{e-r} or \S 1.4 of \cite{xin1})
is
\begin{equation}\label{com1}
\tau(\bar{\phi}\circ \phi)=\bar{\phi}_*(\tau(\phi))+\sum_{j=1}^n B_{\phi_* e_j,\phi_* e_j}(\bar{\phi}).
\end{equation}
When $\bar{\phi}$ is an isometric immersion, we simply write the formula as
\begin{equation}\label{com2}
\tau(\bar{\phi}\circ \phi)=\tau(\phi)+\sum_{j=1}^n B(\phi_* e_j,\phi_* e_j)
\end{equation}
where $B$ is the second fundamental form of $N$ in $\bar{N}$.
\bigskip

\subsection{Necessary and sufficient conditions for LOSs}\label{NSCL}
Throughout our paper, $S^d\subset \R^{d+1}$ be the $d$-dimensional unit sphere,
$g_d$ the canonical metric 
induced by
the inclusion map $i_d: S^d\ra \R^{d+1}$, and $B_d$ the second fundamental form of $S^d$ in $\R^{d+1}$.

Given smooth $f:S^n\rightarrow S^m$ and an acute angle $\th$, let $I_{f,\th}: S^n\ra S^{n+m+1}$
\begin{equation}
I_{f,\th}(x)=(\cos\th\cdot x,\sin\th\cdot f(x))
\end{equation}
be the embedding associated to $f$ and $\th$, and $g:=I_{f,\th}^*g_{n+m+1}$.
We shall study how to have a minimal $I_{f,\th}$ and thus an LOS $M_{f,\th}$.

Let $\mathbf{X}(x)$,
$\mathbf Y_1(x)$ and $\mathbf Y_2(x)$ be the position vectors of
 $I_{f,\th}(x)$ in $\R^{n+m+2}$,
$x$ in $\R^{n+1}$ and $f(x)$ in $\R^{m+1}$ respectively.
Then
\begin{equation}\label{position}
\mathbf{X}(x)=(\cos\th\mathbf{Y}_1(x),\sin\th\mathbf{Y}_2(x)).
\end{equation}

On the one hand,
$\De_g \mathbf X=\tau(\mathbf X)$.
By $\mathbf X=i_{n+m+1}\circ I_{f,\th}$ and \eqref{com2} ,
we have
\begin{equation}\label{x2}
\aligned
\De_g \mathbf{X}&=\tau(\mathbf{X})=\tau(i_{n+m+1}\circ I_{f,\th})
=\tau(I_{f,\th})+\sum_{j=1}^n B_{n+m+1}((I_{f,\th})_* e_j,(I_{f,\th})_* e_j)\\
&=\mathbf{H}-\sum_{j=1}^n \lan (I_{f,\th})_* e_j,(I_{f,\th})_* e_j\ran\mathbf{X}
=\mathbf{H}-\sum_{j=1}^n g(e_j,e_j)\mathbf{X}=\mathbf{H}-n\mathbf{X}.
\endaligned
\end{equation}
Here $\{e_1,\cdots,e_n\}$ is an orthonormal basis of $(T_x S^n,g)$, $\lan \cdot,\cdot\ran$ the Euclidean inner product,
and
$\mathbf H$ the mean curvature field of $(S^n,g)$ in $S^{n+m+1}$.
We remark that $\mathbf H\bot \mathbf {X}$ pointwise.

On the other hand, similarly
for $\mathbf{Y_1}=i_n\circ \mathbf{Id}$ where $\mathbf{Id}$ is the identity map from $(S^n,g)$ to $(S^n,g_n)$
and $\mathbf{Y_2}=i_m\circ f$,
we gain
\begin{equation}\label{y1}
\De_g \mathbf{Y}_1=\tau(\mathbf{Y}_1)=\tau(i_n\circ \mathbf{Id})
=\tau(\mathbf{Id})-2e(\mathbf{Id})\mathbf{Y}_1,
\end{equation}
where $\tau(\mathbf{Id})\bot \mathbf{Y}_1$,
and
\begin{equation}\label{y2}
\De_g \mathbf{Y}_2=\tau(\mathbf{Y}_2)
=\tau(f)-2e(f)\mathbf{Y}_2,
\end{equation}
with $\tau(f)\bot \mathbf Y_2$.
Therefore,
\begin{equation}\label{x1}
\De_g \mathbf{X}
=\Big(\cos\th \big(\tau(\mathbf{Id})-2e(\mathbf{Id}) \mathbf{Y}_1\big),\sin\th \big(\tau(f)-2e(f) \mathbf{Y}_2\big)\Big).
\end{equation}

By comparing \eqref{x2} and \eqref{x1} we obtain
\begin{equation}\label{mean}
\mathbf{H}=\Big(\cos\th \big(\tau(\mathbf{Id})-(2e(\mathbf{Id})-n) \mathbf{Y}_1\big),\sin\th \big(\tau(f)-(2e(f)-n) \mathbf{Y}_2\big)\Big).
\end{equation}
We shall employ this relationship to derive the following characterization of LOS.

\begin{thm}\label{con}

For smooth $f:S^n\rightarrow S^m$ and $\th\in (0,\pi/2)$,
$I_{f,\th}$ is minimal (i.e., $M_{f,\th}$ is an LOS)
if and only if the following conditions hold:
\begin{enumerate}
\item[(a)] $f:(S^n,g)\ra (S^m,g_m)$ is harmonic;
\item[(b)] For each $x\in S^n$ and the singular values $\la_1,\cdots,\la_n$  of
$(f_*)_x: (T_x S^n,g_n)\ra (T_{f(x)} S^m,g_m)$,
$\sum\limits_{j=1}^n \f{1}{\cos^2\th+\sin^2\th \la_j^2}=n.$
\end{enumerate}

\end{thm}

\begin{proof}






 We shall use two equivalent statements of Condition (b):
\begin{enumerate}
\item [(c)] The energy density of $\mathbf{Id}:(S^n,g)\ra (S^n,g_n)$ is $\f{n}{2}$ everywhere.
\item [(d)] The energy density of $f:(S^n,g)\ra (S^m,g_m)$ is $\f{n}{2}$ everywhere.
\end{enumerate}
Let us first show (b)$\Leftrightarrow$(c).
By choosing an orthonormal basis $\{\ep_1,\cdots,\ep_n\}$ of $(T_x S^n,g_n)$
for
\begin{equation}
\lan f_* \ep_j,f_*\ep_k\ran=\la_j^2 \de_{jk},
\end{equation}
and setting
\begin{equation}\label{basis3}
e_j:=\f{1}{\sqrt{\cos^2\th+\sin^2\th\la_j^2}}\ep_j,
\end{equation}
we have
\begin{equation}\aligned
g(e_j,e_k)&=\lan (I_{f,\th})_* e_j,(I_{f,\th})_* e_k\ran\\
&=\lan (\cos\th e_j,\sin\th f_* e_j),(\cos\th e_k,\sin\th f_* e_k)\ran\\
&=\cos^2\th \lan e_j,e_k\ran+\sin^2 \th\lan f_* e_j,f_* e_k\ran\\
&=\de_{jk}.
\endaligned
\end{equation}
This means $\{e_1,\cdots,e_n\}$ form an orthonormal basis of $(T_x S^n,g)$. Here and in the sequel,
we call such $\{\ep_1,\cdots,\ep_n\}$ and $\{e_1,\cdots,e_n\}$ the {\bf S-bases} of
$(S^n,g_n)$ and $(S^n,g)$ for $f$  (w.r.t. $\la_1,\cdots,\la_n$ and $\f{\la_1}{\sqrt{\cos^2\th+\sin^2\th \la_1^2}},\cdots,\f{\la_n}{\sqrt{\cos^2\th+\sin^2\th\la_n^2}}$ respectively).
Then
$$2e(\mathbf{Id})=\sum_{j=1}^n\lan e_j,e_j\ran=\sum_{j=1}^n \f{1}{\cos^2\th+\sin^2\th\la_j^2},$$
and therefore (b)$\Leftrightarrow$(c).
Also note that for a fixed acute angle $\th$
$$\sum\limits_{j=1}^n \f{1}{\cos^2\th+\sin^2\th \la_j^2}=n
\ \ \Leftrightarrow\ \
\sum\limits_{j=1}^n \f{\la_j^2}{\cos^2\th+\sin^2\th \la_j^2}=n.$$
So  (b) and (d) are equivalent as well.

Apparently, $\mathbf{H}=0$ in \eqref{mean} implies $\tau(\mathbf{Id})=\tau(f)=0$ and $e(\mathbf{Id})=e(f)=\f{n}{2}$, i.e., Conditions (a)-(b).
Conversely, under Conditions (a)-(b),
 \eqref{mean} becomes
$\mathbf{H}=(\cos\th\cdot \tau(\mathbf{Id}),0).$
Observing $\mathbf{H}\bot (I_{f,\th})_* (T_x S^n)$,
$$0=\lan \mathbf{H},(I_{f,\th})_* v\ran=\big\lan (\cos\th\cdot \tau(\mathbf{Id}),0),(v,f_*v)\big\ran=\cos\th \lan \tau(\mathbf{Id}),v\ran$$
for every $v\in T_x S^n$,
we know
\begin{equation}\label{Idh}
\tau(\mathbf{Id})=0
\end{equation}
 and therefore $\mathbf{H}=0$.
\end{proof}


\bigskip
\subsection{Characterizations of trivial LOMs}

For an isometric totally geodesic embedding $f:(S^n,g_n)\ra (S^m,g_m)$,
$M_{f,\th}$ is totally geodesic in $S^{n+m+1}$ for arbitrary $\th\in (0,\pi/2)$.
In such case {$f$} is called a {\bf trivial LOM}.

\begin{pro}\label{g1}

For an LOM $f:S^n\ra S^m$, the followings are equivalent:

\begin{enumerate}
\item[(i)] All singular values of $(f_*)_x$ are equal at each $x$.
\item[(ii)] All singular values of $(f_*)_x$ are equal to $1$.
\item[(iii)] $f:(S^n,g_n)\ra (S^m,g_m)$ is an isometric immersion.
\item[(iv)] $f:(S^n,g_n)\ra (S^m,g_m)$ is an isometric totally geodesic embedding.
\item[(v)] For every $\th\in (0,\pi/2)$, $M_{f,\th}$ is totally geodesic.
\item[(vi)] There exists $\th\in (0,\pi/2)$, such that $M_{f,\th}$ is a totally geodesic LOS.

\end{enumerate}

\end{pro}

\begin{proof}

(i)$\Rightarrow$(ii) immediately follows from Condition (b) in Theorem \ref{con};
 (iii)$\Rightarrow$(iv) is a direct corollary
of Condition (a) in Theorem \ref{con} and the {Gauss equations} of submanifold;
and (ii)$\Rightarrow$(iii) and
(iv)$\Rightarrow$(v)$\Rightarrow$(vi)$\Rightarrow$(i) are trivial.
\end{proof}

\begin{cor}\label{cor1}
Let $f:S^n\ra S^m$ be a smooth map.
Then
\begin{itemize}
\item If $n\geq 2$ and $m=1$, then $f$ cannot be an LOM.

\item If $n\leq 2$ and $m\geq n$, then $f$ is an LOM if and only if $f$ is a trivial one.
\end{itemize}
\end{cor}

\begin{proof}
We study each case according to values of $n$ and $m$.

\textbf{Case I. }$n=1$.
 $(f_*)_x$ has only one singular value. By (i) of
Proposition \ref{g1}, $f$ is an LOM if and only if $f$ is a trivial one.

\textbf{Case II. }$n\geq 2$, $m=1$.
If there were one LOM $f$, then by Theorem \ref{con}
$f:(S^n,g)\ra S^1$ is harmonic.
So is its lifting map $\td{f}:(S^n,g)\ra\R$.
But the strong maximal principle forces $\td{f}$ (and therefore $f$) to be constant.
This contradicts (ii) of Proposition \ref{g1}.
So there are no LOMs in the setting.

\textbf{Case III. }$n=2, m\geq 2$.
Assume $f$ is an LOM.
Then by Theorem \ref{con} and \eqref{Idh}
both $\mathbf{Id}: (S^2,g)\ra (S^2,g_2)$ and $f: (S^2,g)\ra (S^m,g_m)$ are harmonic.
 Since every harmonic map from $2$-sphere
 (equipped with arbitrary metric)
 is conformal (see \S I.5 of \cite{s-y}),
there exist an orthonormal basis $\{e_1,e_2\}$ of $(T_x S^2,g)$ with
$$\lan e_1,e_1\ran =\lan e_2,e_2\ran=e(\mathbf{Id})=1,\quad \lan e_1,e_2\ran=0$$
and
$$\lan f_* e_1,f_* e_1\ran=\lan f_* e_2,f_* e_2\ran=e(f)=1,\quad \lan f_* e_1,f_* e_2\ran=0.$$
Hence $f:(S^2,g_2)\ra (S^m,g_m)$ is an isometric immersion. By (iii) of Proposition \ref{g1},
$f:(S^2,g_2)\ra (S^m,g_m)$ is a trivial LOM.
\end{proof}


\bigskip

\subsection{Nontrivial LOMSEs}

In this subsection we focus on nontrivial LOMSEs.
We first establish a useful structure theorem for LOMSEs.

\begin{thm}\label{g2}
For smooth $f:S^n\ra S^m$, the followings are equivalent:
\begin{enumerate}
\item[(i)] $f$ is a nontrivial LOMSE, namely for each $x\in S^n$ all nonzero singular values of $(f_*)_x$ are equal.
\item[(ii)] $f$ is an LOM and has two constant singular values $0$ and $\la>0$
of constant multiplicities $(n-p)$ and $p$ respectively everywhere.
\item[(iii)] There exist a $p$-dimensional Riemannian manifold $(P,h)$, a real number $\la>\sqrt{\f{n}{p}}$,
$\pi:S^n\rightarrow P$ and
 $i:P\rightarrow S^m$,
 such that $f=i\circ \pi$, $\pi:(S^n,g_n)\ra (P,h)$ is a harmonic Riemannian submersion with connected fibers
 and
 $i:(P,\la^2 h)\ra (S^m,g_m)$ an isometric minimal immersion.

\end{enumerate}
Moreover, if $f$ satisfies one of the above, then $M_{f,\th}$ is an LOS
 exactly when
 \begin{equation}\label{th}
\th=\arccos\sqrt{\f{n-p}{n(1-\la^{-2})}}.
\end{equation}
\end{thm}

We shall utilize next two lemmas in proving Theorem \ref{g2}.

\begin{lem}\label{lem2}
Let $(\bar{N}^n,\bar{g}), (N^m,g)$ be Riemannian manifolds.
Assume further $\bar{N}$ is connected and compact, and
$\phi:(\bar{N},\bar{g})\ra (N,g)$ a smooth map
with singular values 0 and 1 of multiplicities $(n-p)$ and $p$ pointwise.
Then there exist a Riemannian manifold $(P^p,h)$, a Riemannian submersion $\pi:(\bar{N},\bar{g})\ra (P,h)$ with connected fibers
 and an isometric
immersion $i:(P,h)\ra (N,g)$, such that $\phi=i\circ \pi$.
\end{lem}

 We save its proof to Appendix \S \ref{App1}.

\begin{lem}\label{lem3}
Given a smooth foliation  $\{K_\a:\a\in \La\}$ of $d$-dimensional submanifolds in a manifold $M^n$,
suppose $g$ and $\td{g}$ are Riemannian metrics on $M$, satisfying:
\begin{enumerate}
\item[(a)] $\exists$ constant $\mu>0$ so that $\td{g}|_{K_\a}=\mu g|_{K_\a}$ for all $\a\in \La$;
\item[(b)] For every $\a\in \La$, $p\in K_\a$, $v\in T_p M$ and $w\in T_p K_\a$,\\
 $\td{g}(v,w)=0$ if and only if $g(v,w)=0$.
\end{enumerate}
Then $K_\a$ is minimal in $(M,\td{g})$ if and only if it is minimal in $(M,g)$.
\end{lem}

\begin{proof}
Let $\n$ be the Levi-Civita connection for $g$.
Then (e.g. see \S 2.3 of \cite{dC})
\begin{equation}\label{gn}\aligned
g(\n_X Y,Z)=\f{1}{2}\Big\{
&
\n_X g(Y,Z)+\n_Y g(Z,X)-\n_Z g(X,Y)\\
&+g(Y,[Z,X])+g(Z,[X,Y])-g(X,[Y,Z])\Big\}
\endaligned
\end{equation}
for vector fields $X,Y,Z$ on $M$.

Denote by $B$ the second fundamental form of $K_\a$ in $(M,g)$ and $\mathbf{H}$ the mean vector field.
From \eqref{gn} we get
\begin{equation*}\aligned
&g(\mathbf{H},\nu)=\sum_{i=1}^d g(B(E_i,E_i),\nu)=\sum_{i=1}^d g(\n_{E_i}E_i,\nu)\\
=&\sum_{i=1}^d\f{1}{2}\Big\{2\n_{E_i}g(E_i,\nu)-\n_{\nu} g(E_i,E_i)+2g(E_i,[\nu,E_i])+g(\nu,[E_i,E_i])\Big\}\\
=&\sum_{i=1}^d g(E_i,[\nu,E_i]).
\endaligned
\end{equation*}
Here $\{E_1,\cdots,E_n\}$ is a local orthonormal frame
such that pointwise the first $d$ terms form an orthonormal basis
of leaves, and $\nu$ a vector field orthogonal to leaves.

Similarly, using symbols $\td{B}$ and $\td{\mathbf{H}}$ for $\td{g}$,
we have
\begin{equation*}
\td{g}(\td{\mathbf{H}},\nu)=\sum_{i=1}^d \mu^{-1}\td{g}(\td{B}(E_i,E_i),\nu)=\mu^{-1}\sum_{i=1}^d \td{g}(E_i,[\nu,E_i])\
=\sum_{i=1}^d g(E_i,[\nu,E_i]).
\end{equation*}
Therefore, $\mathbf{H}=0$ if and only if $\td{\mathbf{H}}=0$.
\end{proof}

\renewcommand{\proofname}{\bf Proof of Theorem \ref{g2}}
\begin{proof}
Under (i),
Condition (b) of Theorem \ref{con} implies
\begin{equation}\label{la}
\la=\sqrt{\f{n\cos^2\th}{p-n\sin^2\th}}\in \left(\sqrt{\f{n}{p}},+\infty\right).
\end{equation}
Since $\la$ varies continuously on $S^n$, both $\la$ and $p$ must be constant.
Hence (i)$\Rightarrow$(ii) and (\ref{th}) hold.

To show (ii)$\Rightarrow$(iii),
note that $f:(S^n,g_n)\ra (S^m,\la^{-2}g_m)$
has singular values $0$ and $1$ of multiplicities $(n-p)$ and $p$.
By Lemma \ref{lem2}, there exist Riemannian manifold $(P^p,h)$, Riemannian submersion $\pi:(S^n,g_n)\ra (P,h)$
with connected fibers and isometric immersion $i:(P,h)\ra (S^m,\la^{-2}g_m)$, such that $f=i\circ \pi$.
Now it
suffices to show that such $\pi$ and $i$ are harmonic.

By Condition (a) of Theorem \ref{con}, $f:(S^n,g)\ra (S^m,g_m)$ is harmonic. So is $f:(S^n,g)\ra (S^m,\la^{-2}g_m)$.
Moreover, (\ref{com2}) leads to
$$0=\tau(f)=\tau(i\circ \pi)=\tau(\pi)+\sum_{j=1}^n B(\pi_* e_j,\pi_* e_j),$$
 where $\{e_1,\cdots,e_n\}$ form an orthonormal basis
 at the considered point w.r.t. $g$
 and
 $B$ the second fundamental form of the immersed $(P,h)$ in $(S^m,\la^{-1}g_m)$.
 Observe that $\tau(\pi)$ and $\sum_{j=1}^n B(\pi_* e_j,\pi_* e_j)$ are
 tangent and normal vectors to $P$ respectively.
 Therefore, $\pi: (S^n,g)\ra (P,h)$ is harmonic, and
\begin{equation}\label{B1}
\sum_{j=1}^n B(\pi_* e_j,\pi_* e_j)=0.
\end{equation}

Assume $\la_1=\cdots=\la_p=\la$ and $\la_{p+1}=\cdots=\la_n=0$.
Choose $\{\ep_1,\cdots,\ep_n\}$ and $\{e_1,\cdots,e_n\}$ to be S-bases of
$(T_x S^n,g_n)$ and
$(T_x S^n,g)$ for $f$ accordingly.
Then 
 $\{\pi_* \ep_1,\cdots,\pi_* \ep_p\}$ give an orthonormal basis of
$(T_{\pi(x)}P,h)$ and $\pi_* \ep_i=0$ for $p+1\leq i\leq n$.
Hence $\sum\limits_{j=1}^p B(\pi_* \ep_j,\pi_* \ep_j)=0$.
By \eqref{basis3},
$i:(P,h)\ra (S^m,\la^{-2}g_m)$ is an isometric minimal immersion.

Next, we show $\pi: (S^n,g_n)\ra (P,h)$ is harmonic.
By the above,
both $\pi: (S^n,g)\ra (P,\mu^2 h)$ with $\mu:=(\cos^2\th+\sin^2\th\la^{2})^{-\f{1}{2}}$
and $\pi: (S^n,g_n)\ra (P,h)$ are Riemannian submersions.
Since $g$ and $g_n$ satisfy Conditions (a)-(b) of Lemma \ref{lem3},
together with Proposition \ref{ER} we gain the harmonicity of $\pi: (S^n,g_n)\ra (P,h)$
from that of $\pi$ w.r.t. $g$.
Thus, (ii)$\Rightarrow$(iii).

Finally, the proof of (iii)$\Rightarrow$(i) is similar to that of (ii)$\Rightarrow$(iii),
where one instead argues
that the minimality of fibers to $g_n$ coincides with that to $g$ by Lemma \ref{lem3}.
\end{proof}

\bigskip

\subsection{LOMSEs of (n,p,k)-type}
Based on Theorem \ref{g2} and the spectrum theory of Laplacian operators, we further divide LOMSEs as follows.

\begin{thm}\label{npk}
Let $f:S^n\ra S^m$ be an LOMSE
with nonzero singular value $\la$ of multiplicity $p$.
Then there exists an integer $k\geq 2$,
such that:

\begin{itemize}

\item
For $i_m\circ f(x)= \big(f_1(x),\cdots,f_{m+1}(x)\big)$  in $\R^{m+1}$,
each {component} $f_i$ is a spherical harmonic function of degree $k$.
\item
$\la=\sqrt{\f{k(k+n-1)}{p}}.$
\item $M_{f,\th}$ is an LOS 
if and only if
\begin{equation}\label{th2}
\th=\arccos\sqrt{\f{1-\f{p}{n}}{1-\f{p}{k(k+n-1)}}}.
\end{equation}
\end{itemize}
We call such $f$ an {\bf LOMSE of (n,p,k)-type}.
\end{thm}

\renewcommand{\proofname}{\it Proof.}
\begin{proof}
By Theorem \ref{g2},
we have
$(P^p,h)$, $\pi:S^n\ra P$ and $i:P\ra S^m$, such that $f=i\circ \pi$, $\pi:(S^n,g_n)\ra (P,h)$
a harmonic Riemannian submersion and $i:(P,h)\ra (S^m,\la^{-2}g_m)$ an isometric minimal immersion.

Denote by $\mathbf Y(y)$ the position vector of $i(y)$ in $\R^{m+1}$ for $y\in P$.
Then $\mathbf Y=i_m\circ i\circ \mathbf{Id}$ where
$\mathbf{Id}:(P,h)\rightarrow(P,\la^2 h)$ is the identity map.
Since  $i:(P,\la^2h)\ra (S^m,g_m)$ is an isometric minimal immersion and $\mathbf{Id}$ totally geodesic,
we gain $\tau(i\circ \mathbf{Id})=0$ and thereby via
\eqref{com2} obtain
\begin{equation}\label{com3}\aligned
&\De_h(\mathbf Y)=\tau(\mathbf Y)=\tau(i_m\circ i\circ \mathbf{Id})
=\tau(i\circ \mathbf{Id})+\sum_{j=1}^p B_m\left((i\circ \mathbf{Id})_*e_j,(i\circ \mathbf{Id})_*e_j\right)\\
=&-\left(\sum_{j=1}^p \lan (i\circ \mathbf{Id})_*e_j,(i\circ \mathbf{Id})_*e_j\ran\right)\mathbf Y=-\left(\sum_{j=1}^p \la^2 h(e_j,e_j)\right)\mathbf Y
=-\la^2 p\cdot\mathbf Y
\endaligned
\end{equation}
where $\{e_1,\cdots,e_p\}$ form an orthonormal basis of $(T_y P,h)$.
For $\big(h_1(y),\cdots,h_{m+1}(y)\big):=\mathbf{Y}(y)$,
\eqref{com3} states precisely
\begin{equation}\label{lap_h}
\De_h(h_j)=-\la^2 p\cdot h_j,\ \ \ \ \text{for } 1\leq j\leq m+1.
\end{equation}

Coupling \eqref{com1} with \eqref{lap_h}, we get
\begin{equation}\aligned
\De_{g_n}(h_j\circ \pi)&=\tau(h_j\circ \pi)=(h_j)_*(\tau(\pi))+\sum_{j=1}^n B_{\pi_* \ep_j,\pi_* \ep_j}(h_j)\\
&=\sum_{j=1}^n \text{Hess}_h(h_j)(\pi_* \ep_j,\pi_* \ep_j)=\De_h(h_j)\circ \pi\\
&=-\la^2 p(h_j\circ \pi),
\endaligned
\end{equation}
where $\{\ep_1,\cdots,\ep_n\}$ form an orthonormal basis of $(T_x S^n,g_n)$
such that
 $\{\pi_* \ep_1,\cdots,\pi_* \ep_p\}$ be an orthonormal basis
of $(T_{\pi(x)}P,h)$ and $\pi_* \ep_{p+1}=\cdots=\pi_* \ep_n=0$.
In other words,
\begin{equation}
\De_{g_n}f_j=-\la^2 p\cdot f_j\qquad \forall 1\leq j\leq m+1.
\end{equation}

The theory of eigenvalues of Laplacian operators on Euclidean spheres
confirms the existence of positive integer $k$ so that
every $f_j$ is a spherical harmonic function of degree $k$ (see \S II.4 of \cite{c})
and $\la^2 p=k(k+n-1)$, i.e.,
\begin{equation}\label{sing1}
\la=\sqrt{\f{k(k+n-1)}{p}}.
\end{equation}
Moreover, $\la>\sqrt{\f{n}{p}}$ forces $k\geq 2$.
Finally, \eqref{sing1} and \eqref{th} give (\ref{th2}).
\end{proof}

Based on Theorem \ref{npk}, several geometric quantities of LOSs or LOCs for LOMSEs of $(n,p,k)$-type
can be expressed explicitly. See Appendix \S \ref{App2} for details.

\begin{cor}\label{cor2}
Let $f$ be an LOMSE of $(n,p,k)$-type, $M_{f,\th}$ and $C_{f,\th}$ the corresponding LOS and LOC. Then
\begin{enumerate}
\item[(A)] All normal planes of $M_{f,\th}$ have a constant acute angle $\a_{n,p,k}$ to a preferred reference plane $Q_0$ (see \S\ref{App2}), with
\begin{equation}\label{angle}
\cos\a_{n,p,k}=\sqrt{\f{1-\f{p}{n}}{1-\f{p}{k(k+n-1)}}}\cdot\left(\f{n-p}{k(k+n-1)-p}\right)^{\f{p}{2}}.
\end{equation}
\item[(B)] The volume of $M_{f,\th}$ is
\begin{equation}\label{volume}
V_{n,p,k}=\left(\f{k(k+n-1)}{n}\right)^{\f{p}{2}}\left(\f{1-\f{p}{n}}{1-\f{p}{k(k+n-1)}}\right)^{\f{n-p}{2}}\om_n,
\end{equation}
where $\om_n$ is the volume of $n$-dimensional unit Euclidean sphere.
\item[(C)] $C_{f,\th}$ is an entire minimal graph with constant Jordan angles relative to $Q_0$.
The Jordan angles of
$C_{f,\th}$ are given by
\begin{equation}\label{JA}
\arccos\sqrt{\f{n-p}{k(k+n-1)-p}},\quad\arccos\sqrt{\f{1-\f{p}{n}}{1-\f{p}{k(k+n-1)}}},\quad 0,
\end{equation}
of multiplicities $p,1,n-p$ respectively. The slope function of $C_{f,\th}$ is identically equal to $W_{n,p,k}:=\sec \a_{n,p,k}$.
\end{enumerate}

\end{cor}

\textbf{Remarks.}
\begin{itemize}
\item
The three original LOMs are LOMSEs of $(2m-1,2m,2)$-type for $m=2,4,8$.
Exact values of acute angles for them were provided in \cite{l-o}.
\item 
E. Calabi \cite{ca} proved that the area of
any minimal $S^2$ in spheres has to be an integer multiple of $2\pi$.
In the case of higher dimension,
the gap phenomenon between
volume of totally geodesic spheres and those of other minimal submanifolds in spheres was discovered by Cheng-Li-Yau \cite{c-l-y}.
(B) says that volumes of our examples are discrete.
It should be interesting to study whether volumes of all compact minimal submanifolds take values discretely.
\item
(C) tells that, although
LOCs derived from LOMSEs
are all of constant Jordan angles and are uncountably many
(cf. Theorem \ref{npk2} and the remarks),
the angles take values in a discrete set.
This gives a partial affirmative answer to Problem 1.1 in \cite{j-x-y5}.
\end{itemize}


Let $f_1:S^{n}\ra S^{m_1}$, $f_2:S^{n}\ra S^{m_2}$ be nontrivial LOMs and $m_1\leq m_2$. If
 there exist isometry $\chi:(S^n,g_n)\ra (S^n,g_n)$ and totally geodesic isometric embedding $\psi:(S^{m_1},g_{m_1})\ra (S^{m_2},g_{m_2})$,
 such that the following diagram commutes
$$\CD
 S^n @>\chi>> S^n  \\
 @Vf_1VV     @VVf_2 V \\
 S^{m_1}  @>\psi>> S^{m_2}
\endCD$$
then $f_1$ and $f_2$ are said to be \textit{equivalent}.
By the virtue of structure theorems on Riemannian submersions from Euclidean spheres and minimal immersions into Euclidean spheres,
we obtain a classification of LOMSEs.

\begin{thm}\label{npk2}
Let $\mathcal{F}_{n,p,k}$ be the set of all equivalence classes of $(n,p,k)$-type LOMSEs.
Then $\mathcal{F}_{n,p,k}$ is nonempty if and only if $k$ is a positive even integer and $(n,p)=(15,8)$, $(2l+1,2l)$ or $(4l+3,4l)$ for some positive integer $l$.
Moreover,

\begin{itemize}

\item If $(n,p)=(2l+1,2l)$, there exists a $1:1$ correspondence between $\mathcal{F}_{2l+1,2l,k}$ and the set of equivalence classes of full isometric minimal immersions (see \cite{c-w} for definitions of `equivalence' and `full')
of $(\mathbb{CP}^{l},\frac{k(k+2l)}{2l}g_{FS})$ into unit Euclidean spheres
where $g_{FS}$ means the Fubini-Study metric.

\item If $(n,p)=(4l+3,4l)$, there exists a $1:1$ correspondence between $\mathcal{F}_{4l+1,4l,k}$ and the set of equivalence classes of full isometric minimal immersions
of $(\mathbb{HP}^{l},\frac{k(k+4l+2)}{4l}g_{ST})$ into unit Euclidean spheres
where  $g_{ST}$  is the standard metric on $\mathbb{HP}^l$ (see \S 3.2 of \cite{b-f-l-p-p} for details).

\item If $(n,p)=(15,8)$, there exists a $1:1$ correspondence between $\mathcal{F}_{15,8,k}$ and the set of equivalence classes of full isometric minimal immersions of
$(S^{8},\frac{k(k+14)}{32}g_{8})$ into unit Euclidean spheres.
\end{itemize}
\end{thm}



\begin{proof}
 By Theorem \ref{g2}, an LOMSE $f=i\circ \pi$ where
 $\pi:(S^n,g_n)\longrightarrow (P,h)$ is a harmonic Riemannian submersion with connected fibers and
 $i:(P,\lambda^2h)\longrightarrow(S^m,g_m)$ an isometric minimal immersion.
 By Wilking's classification theorem \cite{w},
 $\pi$ is among Hopf fibrations:
$(S^{2l+1},g_{2l+1})\longrightarrow (\mathbb{CP}^l,g_{FS})$,
$(S^{4l+3},g_{4l+3})\longrightarrow (\mathbb{HP}^l,g_{ST})$ and
$(S^{15},g_{15})\longrightarrow(S^8,\frac{1}{4}g_8)$.
Therefore,
the set of all equivalence classes of $(n,p,k)$-type LOMSEs $1:1$ corresponds to the set of equivalence classes of full isometric minimal immersions
from $(\mathbb{CP}^l,\frac{k(k+2l)}{2l}g_{FS})$ (when $(n,p)=(2l+1,2l)$),
$(\mathbb{HP}^l,\frac{k(k+4l+2)}{4l}g_{ST})$ (when $(n,p)=(4l+3,4l)$) or $(S^8,\frac{k(k+14)}{32}g_{8})$ (when $(n,p)=(15,8)$), into unit Euclidean spheres. 
By Theorem \ref{npk},
 coordinate functions of $f$ in $\R^{m+1}$ are all spherical harmonic polynomials of degree $k$.
Since
$\pi(x)=\pi(-x)$ for $x\in S^n$, $k$ has to be even for nonempty $\mathcal{F}_{n,p,k}$.

Given $(n,p)=(2l+1,2l)$ and $k=2\kappa$ where $l,\k\in \Bbb{Z}^+$,
we need to explain $\mathcal{F}_{n,p,k}$ is nonempty.
Similar argument holds for other cases.
Let $V_{\kappa}$ be the eigenspace of the Laplace-Beltrami operator of $(\mathbb{CP}^l,\frac{k(k+2l)}{2l}g_{FS})$ corresponding to the $\kappa$-th eigenvalue. It is known (see e.g. \S III.C of \cite{b-g-m}) that $V_\kappa$ is nonempty
and elements in $V_\kappa$ are $S^1$-invariant spherical polynomials of degree $2\kappa$ on $S^{2l+1}$.
Choosing an orthonormal basis $\{f_1,\ldots, f_{m+1}\}$ of $V_\kappa$ w.r.t. the $L^2$-inner product of a normalized measure defined in \cite{c-w,u},
then from Takahashi's Theorem \cite{ta}
 we know that the isometric immersion $i:(\mathbb{CP}^l,\frac{k(k+2l)}{2l}g_{FS})\longrightarrow S^m$ by $x\mapsto (f_1(x),\ldots, f_{m+1}(x))$ is minimal.
This is called the \textit{standard minimal immersion} in \cite{c-w,u}.
Combining Theorems \ref{g2} and \ref{npk} implies that $f:=i\circ \pi$ is a LOMSE of $(2l+1,2l,k)$-type.
Such an LOMSE will be called a \textbf{standard
LOMSE} in the sequel. This completes the proof.
\end{proof}

{\bf Remarks.}
\begin{itemize}
\item
In \cite{x-y-z}
we give explicit expressions for standard LOMSEs of $(2l+1,2l,2)$, $(4l+3,4l,2)$-type.
By rigidity results of E. Calabi \cite{ca}, do Carmo-Wallach \cite{c-w}, N. Wallach \cite{wa}, K. Mashimo \cite{ma1,ma2}
and Ohnita \cite{oh},
our construction exhausts all LOMSEs of $(2l+1,2l,2)$, $(4l+3,4l,2)$-type.
We also show that all LOCs corresponding to $k=2$ are area-minimizing therein.

\item
By Theorem \ref{npk2} and structure properties of minimal immersions
from symmetric spaces into spheres due to do Carmo-Wallach \cite{c-w}, Wallach \cite{wa} and Urakawa \cite{u},
$\mc{F}_{n,p,k}$ is smoothly parameterized by a convex body $L$ in a vector space $W_2$.
By do Carmo-Wallach \cite{c-w}, G. Toth \cite{to} and H. Urakawa  \cite{u},
 $\dim W_2\geq 18$ for $(n,p)=(7,4)$ or $(15,8)$ and $k\geq 8$;
 $\dim W_2\geq 91$ for $(n,p)=(2l+1,2l)$, $l\geq 2$, $k\geq 8$ and $\dim W_2\geq 29007$ for $(n,p)=(11,8)$, $k\geq 8$.

\end{itemize}

\bigskip\bigskip

\Section{On Dirichlet problems related to LOMSEs}{On Dirichlet problems related to LOMSEs}

\subsection{Necessary and sufficient conditions for minimal graphs}

Given smooth $f:S^n\ra S^m$ and smooth $\rho:U\subset (0,\infty)\ra \R$,
we shall study
when submanifold $M_{f,\rho}$ in $\R^{n+m+2}$ of form \eqref{frho1}
 is minimal.

Let $g$ be the induced metric on $M_{f,\rho}$
and $h_r:=I_r^*g$ for $r\in U$ where
$I_r:S^n\ra M_{f,\rho}$
\begin{equation}
x\mapsto (rx,\rho(r)f(x)).
\end{equation}
By $r$ and $\rho$
we mean smooth functions $(rx,\rho(r)f(x))\mapsto r$ and $(rx,\rho(r)f(x))\mapsto \rho(r)$.
From now on
we use the symbol $\n$ for the Levi-Civita connection on $(M_{f,\rho},g)$.
Obviously 
$\n_v r=\n_v \rho=0$ for any $v\in TI_{r_0}(S^n)$.

We derive following characterization for minimality of $M_{f,\rho}$ in terms of $r$ and $\rho$.

\begin{thm}\label{min_graph1}
Assume the above function $\rho>0$.
Then $(M_{f,\rho},g)$ is minimal in $\R^{n+m+2}$ if and only if the following two conditions hold:
\begin{enumerate}
\item[(a)]
For each $r\in U$,
$f:(S^n,h_r)\ra (S^m,g_m)$ is harmonic.
\item[(b)] For each $r\in U$,
$\De_g \rho-2\rho\cdot e(f)=0$ pointwise in $I_r(S^n)$ where $e(f)$ is the energy density of $f:(S^n,h_r)\ra (S^m,g_m)$.
\end{enumerate}

Moreover, Condition (b) has an equivalent description in terms of
singular values $\la_1,\cdots,\la_n$ of $(f_*)_x: (T_x S^n,g_n)\ra
(T_{f(x)}S^m,g_m)$, and that is
\begin{equation}\label{ODE}
\f{\rho_{rr}}{1+\rho_r^2}+\sum_{i=1}^n \f{\f{\rho_r}{r}-\f{\la_i^2 \rho}{r^2}}{1+\f{\la_i^2\rho^2}{r^2}}=0.
\end{equation}

\end{thm}

\begin{proof}
Define $\mathbf{X}:U\times S^n\ra \R^{n+m+2}$ by
$(r,x)\mapsto (r\mathbf{Y}_1(x),\rho(r)\mathbf{Y}_2(x))$
where $\mathbf Y_1(x)\in \R^{n+1}$ and $\mathbf Y_2(x)\in \R^{m+1}$ are position vectors
of $x$ and $f(x)$ respectively.
Then $\mathbf{X}$ is the position function of $M_{f,\rho}$. The tangent plane at $\mathbf{X}(r,x)$ is spanned
 by
$$\pr:=\left(\mathbf{Y}_1(x),\rho_r \mathbf{Y}_2(x)\right)
\text{\ \ \ and \ \ \ }
E_i:=\left(r\ep_i,\rho(r)f_* \ep_i\right)$$
determined by a basis $\{\ep_1,\cdots,\ep_n\}$ of $T_x S^n$.
Easy to see
\begin{equation}
\lan \pr,E_i\ran=0
\text{\ \ \ and \ \ \ }
\lan \pr,\pr\ran=1+\rho_r^2.
\end{equation}

For the mean curvature vector field on $M_{f,\rho}$
\begin{equation}\label{MeanCur}
\mb{H}=\De_g \mb{X}=\big(\De_g(r\mb{Y}_1),\De_g(\rho(r)\mb{Y}_2)\big),
\end{equation}
let us do calculations on its second component.
As $\mathbf{Y}_2(x)$ is 
independent of $r$
on $(M_{f,\rho},g)$,
we get
$$
\Hess\mathbf{Y}_2(\pr,\pr)=\n_{\pr}\n_{\pr}\mathbf{Y}_2-(\n_{\pr}\pr)\mathbf{Y}_2=-(\n_{\pr}\pr)\mathbf{Y}_2.$$
Since
$
\lan \n_{\pr}\pr,E_i\ran=\lan\n_{\pr}(\mathbf{Y}_1(x),\rho_r\mathbf{Y}_2(x)), E_i\ran
=\big\lan (0,\rho_{rr} \mathbf{Y}_2(x)),(r\ep_i,\rho(r) f_* \ep_i)\ran=0
$
for $1\leq i\leq n$,
$\n_{\pr}\pr$ is parallel to $\pr$ and hence $\Hess \mathbf{Y}_2(\pr,\pr)=0$.
Moreover,
in the Riemannian submanifold $(S^n,h_r)$ in $(M_{f,\rho},g)$,
we have
\begin{equation}
\De_g \mathbf{Y}_2=\De_{h_r}\mathbf{Y}_2+\lan \pr,\pr\ran^{-1}\Hess \mathbf{Y}_2(\pr,\pr)=\De_{h_r}\mathbf{Y}_2.
\end{equation}
As in \S \ref{NSCL}, for $f:(S^n,h_r)\ra (S^m,g_m)$ we gain
\begin{equation}
\De_g \mb{Y}_2=\De_{h_r}\mathbf{Y}_2=\tau(\mathbf{Y}_2)=\tau(f)-2e(f)\mb{Y}_2,
\end{equation}
and further,
\begin{equation}\label{La}
\aligned
&\De_g(\rho(r)\mb{Y}_2)\\
=&(\De_g\rho) \mb{Y}_2+\rho \De_g \mb{Y}_2+2\lan \pr,\pr\ran^{-1}\rho_r \n_{\pr}\mb{Y}_2+\sum_{i,j}g^{ij}\n_{E_i}\rho\n_{E_j}\mathbf Y_2\\
=&\rho\cdot\tau(f)+(\De_g \rho-2\rho\cdot e(f))\mb{Y}_2,
\endaligned
\end{equation}
where $(g^{ij})$ is the inverse matrix of $(g_{ij}):=\big(\lan E_i,E_j\ran\big)$.

Therefore, $\mb{H}=0$ implies
$\tau(f)=0$ and $\De_g \rho-2\rho\cdot e(f)=0$.
Conversely, $\tau(f)=0$ and $\De_g \rho-2\rho\cdot e(f)=0$ lead to
$\mb{H}=(\De_g(r\mb{Y}_1),0)$.
Since
$$0=\lan \mb{H},\pr\ran=\big\lan (\De_g(r\mb{Y}_1),0),(\mb{Y}_1,\rho_r\mb{Y}_2)\big\ran=\lan \De_g(r\mb{Y}_1),\mb{Y}_1\ran$$
and
$$0=\lan \mb{H},E_i\ran=\big\lan (\De_g(r\mb{Y}_1),0),(r\ep_i,\rho f_* \ep_i)\big\ran=r\lan \De_g(r\mb{Y}_1),\ep_i\ran\quad \forall 1\leq i\leq n,$$
it follows $\De_g(r\mb{Y}_1)=0$ and thus $\mb{H}=0$.

To show the equivalence of Condition (b) and \eqref{ODE},
we express $e(f)$ and $\De_g \rho$ explicitly.
For an S-basis $\{\ep_1,\cdots,\ep_n\}$ of $(T_x S^n,g_n)$ subject to $\la_1,\cdots,\la_n$,
\begin{equation}
\lan E_i,E_j\ran=\lan (r\ep_i,\rho_r f_* \ep_i),(r\ep_j,\rho f_* \ep_j)\ran=(r^2+\rho^2 \la_i^2)\de_{ij}
\end{equation}
and
\begin{equation}\label{den2}
2e(f)=\sum_{i=1}^n \f{\lan f_* \ep_i,f_* \ep_i\ran}{h_r(\ep_i,\ep_i)}=\sum_{i=1}^n \f{\lan f_* \ep_i,f_* \ep_i\ran}{\lan E_i,E_i\ran}=\sum_{i=1}^n \f{\la_i^2}{r^2+\rho^2\la_i^2}.
\end{equation}
%
Meanwhile,
we have
\begin{equation}\label{La1}\aligned
\De_g \rho(r)&=\rho_r \De_g r+\rho_{rr}|\text{grad}_g r|^2\\
&=\rho_r\left(\f{\text{Hess}_g r(\pr,\pr)}{\lan \pr,\pr\ran}+\sum_{i=1}^n \f{\text{Hess}_g r(E_i,E_i)}{\lan E_i,E_i\ran}\right)+\rho_{rr}|\text{grad}_g r|^2
\endaligned
\end{equation}
where $\text{Hess}_g$ and $\text{grad}_g$
are Hessian and 
gradient operators respectively. 
Using
\begin{equation}\label{grad}
|\text{grad}_g r|^2=\lan \pr,\pr\ran^{-1}=\f{1}{1+\rho_r^2},
\end{equation}
\begin{equation}\label{hess1}\aligned
&\text{Hess}_g r(\pr,\pr)=\n_{\pr}dr(\pr)-dr(\n_{\pr}\pr)\\
=&-\f{\lan\n_{\pr}\pr,\pr\ran}{\lan \pr,\pr\ran}dr(\pr)=-\f{\n_{\pr} \lan \pr,\pr\ran}{2 \lan\pr,\pr\ran}
=-\f{\rho_r\rho_{rr}}{1+\rho_r^2}
\endaligned
\end{equation}
and
\begin{equation}\label{hess2}\aligned
&\text{Hess}_g r(E_i,E_i)=\n_{E_i}dr(E_i)-dr(\n_{E_i}E_i)\\
=&-\f{\lan \n_{E_i}E_i,\pr\ran}{\lan \pr,\pr\ran}dr(\pr)=\f{\lan \n_{E_i}\pr,E_i\ran}{\lan \pr,\pr\ran}
=\f{r+\rho\rho_r\la_i^2}{1+\rho_r^2},
\endaligned
\end{equation}
we can simplify \eqref{La1} to be
\begin{equation}\label{La3}
\De_g \rho
=\f{\rho_{rr}}{(1+\rho_r^2)^2}+\sum_{i=1}^n\f{\rho_r(r+\rho\rho_r \la_i^2)}{(1+\rho_r^2)(r^2+\rho^2\la_i^2)}
\end{equation}
which together with
\eqref{den2}
shows that Condition (b) and \eqref{ODE} are the same.
\end{proof}

\textbf{Remark. }
Let $\rho:U\ra \R$ be smooth (not requiring $\rho>0$) so that $M_{f,\rho}$ is minimal.
Set
$Z=\{r:\rho(r)=0\}$.
Then, by \eqref{La}, $f:(S^n,h_r)\ra (S^m,g_m)$ is harmonic for $r\in U- Z$.
If $Z$ has interior points, the analyticity forces $\rho\equiv 0$.
For $\rho\not\equiv 0$,
since the tension field is smoothly depending on metric,
 the harmonicity of $f:(S^n,h_r)\ra (S^m,g_m)$ holds for $r\in Z$.
Therefore, `$\rho>0$' in the theorem can be replaced by `$\rho\not\equiv 0$'.

\bigskip

In the sequel 
we first establish a simple version of Theorem \ref{min_graph1} for LOMSEs
and then obtain several important applications.

\subsection{Entire minimal graphs associated to LOMSEs}\label{EMG}

For an LOMSE $f:S^n\ra S^m$ of $(n,p,k)$-type
we know in Theorem \ref{npk} that
\begin{equation}\label{sing2}
\la=\sqrt{\f{k(k+n-1)}{p}}
\end{equation}
is the nonzero singular value 
and
from Theorem \ref{g2} that
$f=i\circ \pi$ where $\pi:(S^n,g_n)\ra (P,h)$ is a harmonic Riemannian submersion and $i:(P,h)\ra (S^m,\la^{-2}g_m)$ an isometric minimal immersion.

Let $x\in S^n$,
$\la_1=\cdots=\la_p=\la$ and $\la_{p+1}=\cdots=\la_n=0$.
Then under an S-basis $\{\ep_1,\cdots,\ep_n\}$ of $(T_x S^n,g_n)$
for $f$
subject to $\la_1,\cdots,\la_n$,
\begin{equation}\label{hr}
h_r(\ep_i,\ep_j)=(r^2+\rho^2\la_i^2)\de_{ij}=\left\{\begin{array}{ll}
0 & i\neq j,\\
r^2+\rho^2\la^2 & 1\leq i=j\leq p,\\
r^2 & p+1\leq i=j\leq n.
\end{array}\right.
\end{equation}
Then
$\pi:(S^n,h_r)\ra (P,\mu h)$ is a Riemannian submersion and $i:(P,\mu h)\ra (S^m,\mu\la^{-2}g_m)$ an isometric minimal immersion,
where $\mu:=r^2+\rho^2\la^2$.
Further, by Lemma \ref{lem3}
we gain the harmonicity of $\pi:(S^n,h_r)\ra (P,\mu h)$
as in the proof of Theorem \ref{npk}.
Employing 
\eqref{com2},
we know that $f=i\circ \pi:(S^n,h_r)\rightarrow(S^m,g_m)$ is always harmonic for $r\in U$,
and hence simplify Theorem \ref{min_graph1} for LOMSEs.

\begin{thm}\label{min_graph2}
Given an LOMSE $f:S^n\ra S^m$ and smooth $\rho:U\subset (0,+\infty)\ra \R$,
$M_{f,\rho}$ is minimal in $\R^{n+m+2}$ if and only if
\begin{equation}\label{ODE1}
\f{\rho_{rr}}{1+\rho_r^2}+\f{(n-p)\rho_r}{r}+\f{p(\f{\rho_r}{r}-\f{\la^2\rho}{r^2})}{1+\f{\la^2\rho^2}{r^2}}=0.
\end{equation}

\end{thm}

\textbf{Remark.}
Description \eqref{ODE1}
for $H^{2m-1,m}$ where $m=2,4$ or $8$
was first
found
 by Ding-Yuan \cite{d-y} based on special symmetries of Hopf maps.
It should be pointed out that our argument is also applicable to non-equivariant $f$.


Let us analyze the ODE (\ref{ODE1}).
As in \cite{d-y}, set
\begin{equation}\label{varphi}
\varphi:=\f{\rho}{r},\quad t:=\log r.
\end{equation}
Since
\begin{equation}\label{rhor}
\rho_r=(e^t\varphi)_t \f{dt}{dr}=\varphi_t+\varphi
\end{equation}
and
\begin{equation}
\rho_{rr}=(\varphi_t+\varphi)_t \f{dt}{dr}=\f{\varphi_{tt}+\varphi_t}{r},
\end{equation}
we can rewrite \eqref{ODE1} as
\begin{equation}
\f{\varphi_{tt}+\varphi_t}{1+(\varphi_t+\varphi)^2}+(n-p)(\varphi_t+\varphi)+\f{p(\varphi_t+\varphi-\la^2\varphi)}{1+\la^2 \varphi^2}=0.
\end{equation}
By introducing
\begin{equation}\label{psi}
\psi:=\varphi_t,
\end{equation}
we transform \eqref{ODE1} to
an autonomous system
\begin{equation}\label{ODE2}
\aligned
\left\{\begin{array}{ll}
\varphi_t=\psi,\\
\psi_t=-\psi-\Big[\big(n-p+\f{p}{1+\la^2\varphi^2}\big)\psi+\big(n-p+\f{(1-\la^2)p}{1+\la^2\varphi^2}\big)\varphi\Big]\big[1+(\varphi+\psi)^2\big].
\end{array}\right.
\endaligned
\end{equation}
So
$\g: t\mapsto (\varphi(t),\psi(t))$ satisfies \eqref{ODE2} if and only if
$\g$ is an integral curve of vector field
$X:=(X_1,X_2)$
where
\begin{equation}\label{X12}\aligned
\left\{\begin{array}{ll}
X_1&=\psi,\\
X_2&=-\psi-\Big[\big(n-p+\f{p}{1+\la^2\varphi^2}\big)\psi+\big(n-p+\f{(1-\la^2)p}{1+\la^2\varphi^2}\big)\varphi\Big]\big[1+(\varphi+\psi)^2\big].
\end{array}\right.
\endaligned
\end{equation}
It can be seen that $X$ has exactly 3 zero points $(0,0)$ and $(\pm\varphi_0,0)$, where
\begin{equation}\label{varphi0}
\varphi_0:=\sqrt{\f{p-n\la^{-2}}{n-p}}.
\end{equation}
Since $X$ is symmetric about the origin (i.e., $X(-\varphi, -\psi)=-X(\varphi,\psi)$), we shall only focus on the half plane $\varphi\geq0$.

\textbf{Remark. }
The zero point $(0,0)$ 
in fact stands for the coordinate $(n+1)$-plane.
The other zero point $(\varphi_0,0)$ 
corresponds to
$\rho(r)=\varphi_0 r$
that gives a Lipschitz solution
$$F_{f,\rho}(y)=\left\{\begin{array}{cc}
\varphi_0|y|f(\f{y}{|y|}) & y\neq 0\\
0 & y=0
\end{array}\right.$$
to the minimal surface equations.
It is easy to see that $\varphi_0=\tan \th$.

The linearization of the system \eqref{ODE2} at $(0,0)$ is
\begin{equation}
\left(\begin{array}{c}\varphi_t\\ \psi_t\end{array}\right) =A\left(\begin{array}{c}\varphi\\ \psi\end{array}\right)
\end{equation}
where
\begin{equation}\label{A}
A=\left(\begin{array}{cc}
0 & 1\\
\la^2p-n & -n-1
\end{array}
\right)=
\left(\begin{array}{cc}
0 & 1\\
k(k+n-1)-n & -n-1
\end{array}
\right).
\end{equation}
Through calculations the eigenvalues of $A$  are
\begin{equation}\label{la12}
\mu_1=k-1,\qquad \mu_2=-n-k,
\end{equation}
with eigenvectors
\begin{equation}\label{V12}
V_1:=(1, \mu_1)^T,\quad V_2:=(1, \mu_2)^T,
\end{equation}
respectively.
Hence  $(0,0)$ is a saddle critical point.

The linearization of the system \eqref{ODE2} at $(\varphi_0,0)$ is
\begin{equation}
\left(\begin{array}{c}(\varphi-\varphi_0)_t\\ \psi_t\end{array}\right)=B\left(\begin{array}{c}\varphi-\varphi_0\\ \psi\end{array}\right)
\end{equation}
where
\begin{equation}
B=\left(\begin{array}{cc}
0 & 1\\
a & b
\end{array}
\right),
\end{equation}
and
\begin{equation}\label{ab}\aligned
a&:=\f{2\la^2(1-\la^2)p\varphi_0^2(1+\varphi_0^2)}{(1+\la^2\varphi_0^2)^2}=2n\Big(\f{n}{k(k+n-1)}-1\Big),\\
b&:=-1-(n-p+\f{p}{1+\la^2\varphi_0^2})(1+\varphi_0^2)=-n-1.
\endaligned
\end{equation}
For the eigenvalues $\mu_3,\mu_4$ of $B$,
we have
 $$\mu_3+\mu_4=\text{tr }B=b<0,\ \ \ \mu_3\mu_4=|B|=-a>0 \text{\ \ \ and}$$
\begin{equation*}
(\mu_3-\mu_4)^2=(\mu_3+\mu_4)^2-4\mu_3\mu_4=b^2+4a
=n^2-6n+1+\f{8n^2}{k(k+n-1)}.
\end{equation*}
When $n=3$, $k\geq 4$ or $n=5,k\geq 6$,
$\{\mu_3,\mu_4\}$ become a pair of conjugate complex numbers with negative real part;
while in other cases, both $\mu_3$ and $\mu_4$ are negative real numbers. Therefore
\begin{enumerate}

\item[(I)] If $(n,p,k)=(3,2,2),(5,4,2),(5,4,4)$ or $n\geq 7$, $(\varphi_0,0)$ is a stable center of (\ref{ODE2});

\item[(II)] If $(n,p)=(3,2)$, $k\geq 4$ or $(n,p)=(5,4),k\geq 6$, $(\varphi_0,0)$ is a stable spiral point of (\ref{ODE2}).

\end{enumerate}

Using proceding local analysis,
we are able to
establish
the existence of a nontrivial bounded solution in both cases.
A technique point is to construct suitable barrier functions.
Since the proofs are a bit long and subtle, we leave them in Appendices \S \ref{App3}-\ref{App4}.

\begin{pro}\label{case1}
If $(n,p,k)=(3,2,2),(5,4,2),(5,4,4)$ or $n\geq 7$, then there exists a smooth solution $t\in \R\mapsto (\varphi(t),\psi(t))$ to (\ref{ODE2}), with properties
\begin{itemize}
\item $\lim\limits_{t\ra -\infty}(\varphi(t),\psi(t))=(0,0)$;
\item $\varphi(t)=O(e^{(k-1) t})$ and $\psi(t)=O(e^{(k-1) t})$ as $t\ra -\infty$;
\item $\lim\limits_{t\ra +\infty}(\varphi(t),\psi(t))=(\varphi_0,0)$;
\item $t\mapsto \varphi(t)$ is a strictly increasing function;
\item $\psi(t)> 0$ for every $t\in \R$.
\end{itemize}
$$\includegraphics[scale=0.65]{F1.eps}$$

\end{pro}

\begin{pro}\label{case2}
If $(n,p)=(3,2)$, $k\geq 4$ or $(n,p)=(5,4)$, $k\geq 6$, then there exist a smooth solution $t\in \R \mapsto (\varphi(t),\psi(t))$ to (\ref{ODE2})
and a strictly increasing sequence $\{T_i:i\in \Bbb{Z}^+\}$ in $\R$, such that
\begin{itemize}
\item $\lim\limits_{t\ra -\infty}(\varphi(t),\psi(t))=(0,0)$;
\item $\varphi(t)=O(e^{(k-1) t})$ and $\psi(t)=O(e^{(k-1) t})$ as $t\ra -\infty$;
\item $\lim\limits_{t\ra +\infty}(\varphi(t),\psi(t))=(\varphi_0,0)$;
\item $\lim\limits_{i\ra \infty}T_i=+\infty$;
\item $\psi(T_i)=0$ for all $i\in \Bbb{Z}^+$;
\item With $\varphi_i:=\varphi(T_i)$, $\{\varphi_{2m-1}:m\in \Bbb{Z}^+\}$
is  strictly decreasing
and
$\{\varphi_{2m}:m\in \Bbb{Z}^+\}$
is strictly increasing
with the common limit $\varphi_0$;
\item $\psi(t)>0$ for $t\in (-\infty,T_1)\cup \left(\bigcup\limits_{m\in \Bbb{Z}^+}(T_{2m},T_{2m+1})\right)$;
\item $\psi(t)<0$ for $t\in \bigcup\limits_{m\in \Bbb{Z}^+}(T_{2m-1},T_{2m})$;
\item $(\varphi+\psi)(t)>0$ for all $t\in \R$.
\end{itemize}
Namely, the orbit of this solution tends to the saddle point $(0,0)$ as $t\ra -\infty$ and spins around the spiral point $(\varphi_0,0)$
as $t\ra +\infty$.
$$\includegraphics[scale=0.65]{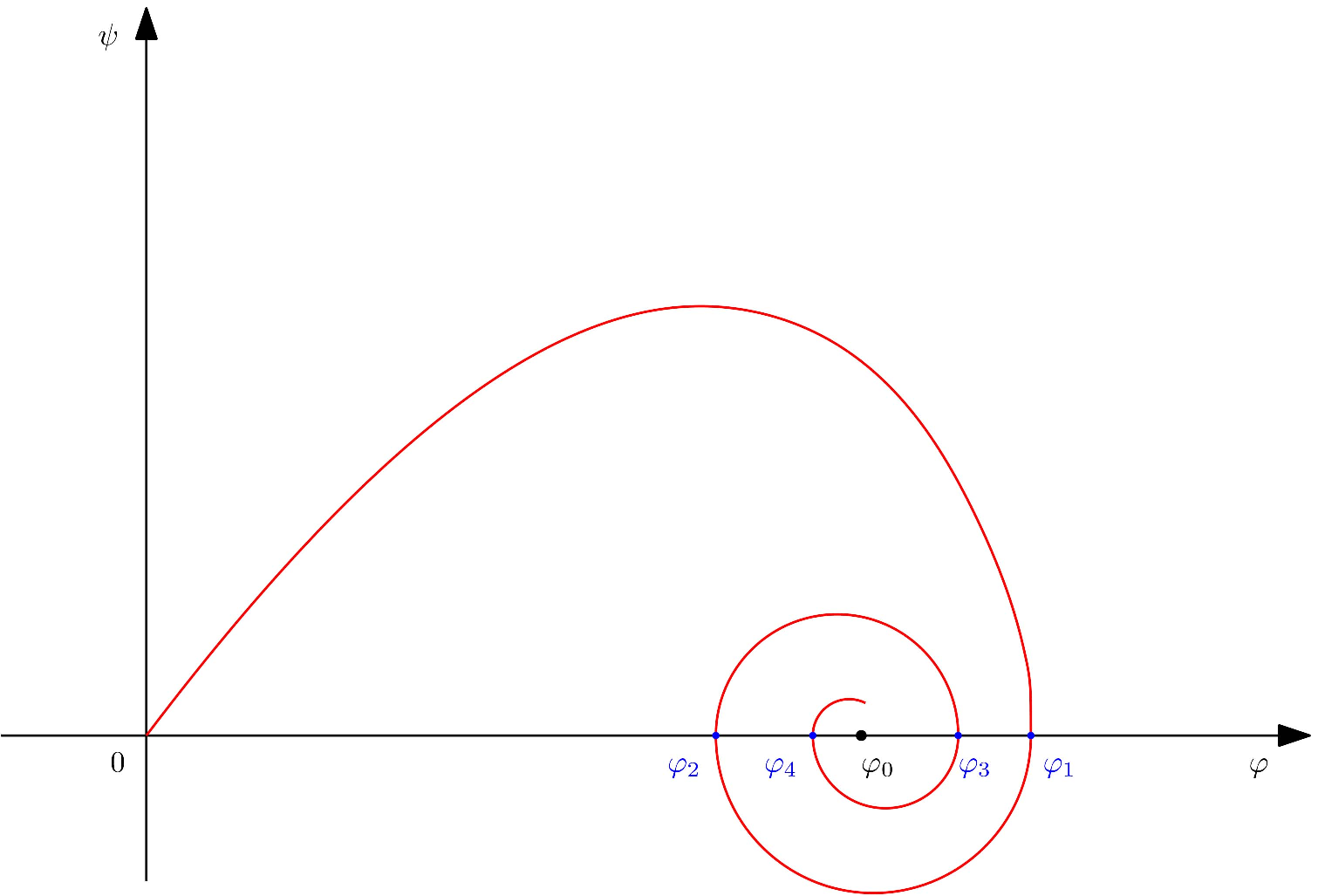}$$
\end{pro}

Based on Propositions \ref{case1}-\ref{case2},
we find entire minimal graphs associated to LOMSEs
as follows.

\begin{thm}\label{graph}
For an LOMSE $f$, there exists a smooth function $\rho$ on $(0,\infty)$
such that
\begin{equation}\label{expF}
F_{f,\rho}(y)=\left\{\begin{array}{cc}
\rho(|y|)f(\f{y}{|y|}) & y\neq 0,\\
0 & y=0.
\end{array}\right.
\end{equation}
gives an analytic entire minimal graph with $C_{f,\th}$, the LOC associated to $f$, as its tangent cone at infinity.
\end{thm}

\begin{proof}
Let $t\in \R\mapsto (\varphi(t),\psi(t))$ be the solution to (\ref{ODE2}) in Proposition \ref{case1}
or \ref{case2}.
Then
\begin{equation}
\rho(r)=r\cdot \varphi(\log r),\text{\ \ for \ } r\in (0,+\infty)\mapsto \rho(r)
\end{equation}
satisfies \eqref{ODE1}.
By Theorem \ref{min_graph2}, $M_{f,\rho}$
is a minimal submanifold in $\R^{n+m+2}$.
Moreover, as $r\ra 0$,
\begin{equation}
\rho(r)=r\cdot \varphi(\log r)=O(r^{k}),
\end{equation}
\begin{equation}
\rho_r=\varphi(\log r)+\psi(\log r)=O(r^{k-1}).
\end{equation}
Hence $F_{f,\rho}$ is $C^1$ 
and further (by Theorem 6.8.1 in \cite{mo})
analytic through the origin.

In addition, by Propositions \ref{case1}-\ref{case2}, $\varphi(t)\ra \varphi_0=\tan{\th}$ as $t\ra +\infty$.
Therefore, the LOC $C_{f,\th}$ is the unique tangent cone of the graph of $F_{f,\rho}$ at infinity.
\end{proof}

\bigskip

\subsection{Non-uniqueness and non-minimizing of minimal graphs}
The amusing spiral asymptotic behavior of the solutions in
Proposition \ref{case2}
produces following interesting corollaries.
They exhibit the non-uniqueness of analytic solutions to the corresponding Dirichlet problem
and the non-minimizing property of those LOCs.

\begin{cor}\label{uni}
For an LOMSE $f$ of $(n,p,k)$-type with $(n,p)=(3,2)$, $k\geq 4$ or $(n,p)=(5,4)$, $k\geq 6$,
there exist infinitely many analytic solutions to the Dirichlet problem for boundary data $f_{\varphi_0}$.
\end{cor}

\begin{proof}
For the solution $t\in \R\mapsto (\varphi(t),\psi(t))$ to \eqref{ODE2} in Proposition \ref{case2},
define $\{t_i\}$ to be the increasing sequence
for
$\varphi(t_i)=\varphi_0$.
Set $d_i=e^{t_i}$ and recall
\begin{equation}\label{Fd_i}
F_{f,\rho_{d_i}}(y)=\f{1}{d_i}F_{f,\rho}(d_i\cdot y),\quad \text{ for } y\in \mathbb D^{n+1} \text{ and } i\in\Bbb{Z}^+.
\end{equation}
Since the minimality is rescaling invariant,
$\{F_{f,\rho_{d_i}}:i\in \mathbb Z^+\}$ give
infinitely many analytic solutions to the minimal surface equations,
with (see \eqref{expF})
\begin{equation}
F_{f,\rho_{d_i}}(x)=\frac{\rho(d_i)}{d_i}f(x)=\varphi(t_i)f(x)=\varphi_0 \cdot f(x),\quad \text{ for } x\in \p \mathbb D^{n+1}.
\end{equation}
Hence we accomplish the proof.
\end{proof}

\textbf{Remark.}
 Similarly, {\it for each $\varphi\in [0,\varphi_1]$, there exists at least one analytic solution to
the Dirichlet problem for $f_\varphi$;
and moreover, for $\varphi\in [\varphi_2,\varphi_1)$
analytic solutions are not unique.}

\bigskip

\begin{cor}\label{non-min}
For an LOMSE $f$ of $(n,p,k)$-type with $(n,p)=(3,2)$, $k\geq 4$ or $(n,p)=(5,4)$, $k\geq 6$,
the LOC $C_{f,\th}$ is non-minimizing.
\end{cor}

\begin{proof}
Let $M$ be the graph of $F_{f,\rho}$.
Then the density function $\Th: \R^+\ra \R$ of $M$ centered at the origin given by
\begin{equation}\label{density1}
\Th(R)=\f{\text{Vol}\big(M\cap \Bbb{D}^{n+m+2}(R)\big)}{\om_{n+1}R^{n+1}},
\end{equation}
where $\om_{n+1}$ denotes the volume of unit ball in $\R^{n+1}$.

Denote by $M_i$ the graph of $F_{f,\rho_{d_i}}$ in \eqref{Fd_i}
and $\Th_i:=
\Th\Big(\sqrt{d_i^2+\rho(d_i)^2}\Big)$.
Then

\begin{equation}
\Th_i=
\f{d_i^{n+1}\text{Vol}(M_i)}{\om_{n+1}\big(\sqrt{d_i^2+\rho(d_i)^2}\big)^{n+1}}
=\f{\text{Vol}\left(M_i\right)}{\om_{n+1}\big(\sqrt{1+
\tan^2\th
}\big)^{n+1}}.
\end{equation}
By the monotonicity theorem for minimal submnaifolds (see e.g. \cite{c-m,fe}),
these quantities increasingly approach the density
$\Th_\infty$
of $C_{f,\rho}$ $-$ the tangent cone of $M$ at infinity,
i.e.,
\begin{equation}
\Th_1\leq \cdots\leq \Th_k\cdots \ra \ \Th_\infty=\f{\text{Vol}\Big(C_{f,\th}\cap \Bbb{D}^{n+m+2}\big(\sqrt{1+\tan^2\th}\big)\Big)}{\om_{n+1}\big(\sqrt{1+\tan^2\th}\big)^{n+1}}.
\end{equation}
If
$\Th_1=\cdots=\Th_\infty$, then
$M$ is a cone.
Since it is not the case,
we have $\Th_1<\Th_0$
and
$$
\text{Vol}\big(M_1\big)
<
\text{Vol}\Big(C_{f,\th}\cap \Bbb{D}^{n+m+2}\big(\sqrt{1+\tan^2\th}\big)\Big).
$$
As
$$
\p\big(M_1\big)
=
\p \Big(C_{f,\th}\cap \Bbb{D}^{n+m+2}\big(\sqrt{1+\tan^2\th}\big)\Big),
$$
we conclude that
$C_{f,\th}$ is not area-minimizing.
\end{proof}

\bigskip\bigskip

\Section{Appendix}{Appendix}
\subsection{Proof of Lemma \ref{lem2}}\label{App1}
Assume $\phi:(\bar{N},\bar g)\rightarrow (N,g)$ and $x\in \text{Im}(\phi)\subset N$.
By the constant rank theorem (see e.g. \S II.7 of \cite{bo})
and the compactness of $\bar{N}$,
the fiber $\phi^{-1}(x)$ over $x$ is a compact
submanifold of $\bar{N}$ with finitely many connected components.
Given $\bar{x}\in \phi^{-1}(x)$, denote by
$[\bar{x}]$ the connected component
of $\phi^{-1}(x)$ containing $\bar{x}$
and set
\begin{equation}
P:=\{[\bar{x}]:\bar{x}\in \bar{N}\}.
\end{equation}
Define
\begin{equation}
\pi(\bar{x})=[\bar{x}]\ \text{\ \ and\ \ \ } i([\bar{x}])=x.
\end{equation}
Then each fiber of $\pi$ is connected and $\phi=i\circ \pi$.

Let $\bar{d}$ and $d$ be the intrinsic distance functions on $(\bar{N},\bar{g})$ and $(N,g)$, respectively,
and $d_H$ the \textit{Hausdorff distance function} on $P$, i.e.
\begin{equation}\label{dh1}
d_H([\bar{x}_0],[\bar{y}_0])=\max\{\sup_{\bar{x}\in [\bar{x}_0]}\inf_{\bar{y}\in [\bar{y}_0]}\bar{d}(\bar{x},\bar{y}),
\sup_{\bar{y}\in [\bar{y}_0]}\inf_{\bar{x}\in [\bar{x}_0]}\bar{d}(\bar{y},\bar{x})\}<+\infty.
\end{equation}
Then $(P,d_H)$
 becomes
 a metric space
with
induced
topology.

Given $[\bar{x}_0],[\bar{y}_0]\in P$
where
representatives
$\bar{x}_0$ and $\bar{y}_0$ are chosen so that
$$\bar{d}(\bar{x}_0,\bar{y}_0)=\bar{d}([\bar{x}_0],[\bar{y}_0]):=\inf\{\bar{d}(\bar{x},\bar{y}):\bar{x}\in [\bar{x}_0],\bar{y}\in [\bar{y}_0]\},$$
let $\bar{\xi}:[0,1]\ra \bar{N}$ be a shortest
geodesic from $\bar{x}_0$ to $\bar{y}_0$ and $\xi:=\phi\circ \bar{\xi}$.
Due to the assumption on singular values, $(\phi_*)_{\bar{x}}: \left((\ker(\phi_*)_{\bar{x}})^\bot,\bar{g}\right)\subset (T_{\bar{x}} \bar{N},\bar{g})\ra (T_{x}N,g)$
is an isometric embedding for each $\bar{x}\in \bar{N}$. Then for each $\bar{x}\in [\bar{x}_0]$, there exists
a unique
 smooth curve $\bar{\xi}_{\bar{x}}:[0,1]\ra \bar{N}$
 such that $\phi\circ \bar{\xi}_{\bar{x}}=\xi$, $\bar{\xi}_{\bar{x}}(0)=\bar{x}$ and $\bar{\xi}_{\bar{x}}'(t)$
 perpendicular to 
  $[\bar{\xi}_{\bar{x}}(t)]$.
  Denote
 \begin{equation}
 \Phi(\bar{x})=\bar{\xi}_{\bar{x}}(1).
 \end{equation}
 Since $\bar{\xi}_{\bar{x}}$ smoothly dependents on $\bar{x}$ and $\text{Length}(\bar{\xi}_{\bar{x}})=\text{Length}
 (\xi)=\text{Length}(\bar{\xi})$, we have
\begin{enumerate}
\item[(A)] $\Phi$ is a diffeomorphism between $[\bar{x}_0]$ and $[\bar{y}_0]$;
\item[(B)] $\bar{d}(\bar{x},\Phi(\bar{x}))=\bar{d}([\bar{x}_0],[\bar{y}_0])=d_H([\bar{x}_0],[\bar{y}_0])$ for each $\bar{x}\in [\bar{x}_0]$.
\end{enumerate}

By the constant rank theorem, the compactness of $\bar{N}$ guarantees the existence of a positive constant $\de$ so that:
\begin{itemize}
\item[{}] For each $\bar{x}\in \bar{N}$ with $x:=\phi(\bar{x})$, $\phi(\td{B}_\de(\bar{x}))$ is a $p$-dimensional embedded submanifold of $N$, and $\td{B}_\de(\bar{x})\cap \phi^{-1}(x)\subset [\bar{x}]$, where $\td{B}_r(\bar{x})$ is the geodesic ball centered at $\bar{x}$ and of radius $r$.
\end{itemize}

Denote by $B_r([\bar{x}])\subset P$ the ball of radius $r$ centered at $[\bar{x}]$.
Based on (A)-(B), we can derive the followings:
\begin{enumerate}
\item [(C)] $i|_{B_{\de/2}([\bar{x}])}$ is injective;
\item [(D)] $i(B_{\de/2}([\bar{x}]))=\phi(\td{B}_{\de/2}(\bar{x}))$ is a $p$-dimensional embedded submanifold of $N$.
\end{enumerate}
 Therefore, we can
 endow $P$ with
a differential structure so that both $i$ and $\pi$ are smooth. 
Moreover,
under
$h:=i^*g$,
$\pi:(\bar{N},\bar{g})\rightarrow(P,h)$ is a Riemannian submersion
and $i:(P,h)\rightarrow(N,g)$ an isometric immersion.
It is worth noting that $d_H$
 is just the intrinsic distance function on $(P,h)$. This completes the proof of Lemma \ref{lem2}.

 \bigskip

\subsection{Proof of Corollary \ref{cor2}}\label{App2}

Suppose $f:S^n\ra S^m$ is an LOM with singular values $\la_1,\cdots,\la_n$ at $x\in S^n$.
Let
$\{\ep_1,\cdots,\ep_n\}$ and $\{e_1,\cdots,e_n\}$ be corresponding S-bases of $(T_x S^n,g_n)$ and $(T_x S^n,g)$, respectively.
Set
\begin{equation}\label{basis4}
E_j=(I_{f,\th})_* e_j=\f{(\cos\th \ep_j,\sin\th f_* \ep_j)}{\sqrt{\cos^2\th+\sin^2\th \la_j^2}},\qquad \forall 1\leq j\leq n.
\end{equation}
Then $\{E_1,\cdots,E_n\}$ form an orthonormal basis of $T_{I_{f,\th}(x)} M_{f,\th}$, and
$$*(\nu_1\w \cdots\w \nu_{m+1})=\mathbf{X}\w E_1\w \cdots\w E_n$$
where $*$ is the {\it Hodge star} operator and $\{\nu_1,\cdots,\nu_{m+1}\}$ is an oriented orthonormal basis of the normal plane $N_{I_{f,\th}(x)} M_{f,\th}$.

Let $\{\ep_{n+2},\cdots,\ep_{n+m+2}\}$ be an oriented orthonormal basis of $Q_0:=\{x_1=\cdots=x_{n+1}=0\}$
and
$\a$ the angle between $Q_0$ and $N_{I_{f,\th}(x)} M_{f,\th}$.
Then
\begin{equation}
\aligned
\cos\a&=\lan \nu_1\w \cdots\w \nu_{m+1},\ep_{n+2}\w\cdots\w \ep_{n+m+2}\ran\\
&=\lan *(\nu_1\w \cdots\w \nu_{m+1}), *(\ep_{n+2}\w\cdots\w \ep_{n+m+2})\ran\\
&=\lan \mathbf{X}\w E_1\w \cdots \w E_n,\mathbf{Y}_1\w \ep_1\w \cdots\w \ep_n\ran\\
&=\left|\begin{array}{cccc}
\lan \mathbf{X},\mathbf Y_1\ran & \lan \mathbf{X},\ep_1\ran & \cdots & \lan\mathbf{X},\ep_n\ran\\
\lan E_1,\mathbf Y_1\ran & \lan E_1,\ep_1\ran & \cdots & \lan E_1,\ep_n\ran\\
& \cdots & &\\
\lan E_n,\mathbf Y_1\ran & \lan E_n,\ep_1\ran & \cdots & \lan E_n,\ep_n\ran
\end{array}\right|\\
&=\cos\th \prod_{j=1}^n \f{\cos\th}{\sqrt{\cos^2\th +\sin^2\th \la_j^2}}
\endaligned
\end{equation}
By applying Theorem \ref{npk}, we obtain \eqref{angle}.

Note that on $(S^n,g)$
the volume form
\begin{equation}\label{dV}
dV=\sqrt{\det\big(g(\ep_j,\ep_k)\big)}\ \ep^*_1\w\cdots\w \ep^*_n=\prod_{j=1}^n \sqrt{\cos^2\th+\sin^2\th \la_j^2}\ \ep^*_1\w \cdots\w \ep^*_n.
\end{equation}
By integration over $S^n$, the fomula \eqref{volume} follows.

It is easy to see that, at $y=tI_{f,\theta}(x)$ for $t>0$,
$X, E_1,\cdots, E_n$ are precisely the angle directions (see \cite{wo} for definition) of $T_yC_{f,\theta}$ relative to $Q_0^\perp$,
with Jordan angles
\begin{equation}
\th_0=\th\quad \text{and}\quad \th_i=\arccos\left(\f{\cos\th}{\cos^2\th+\sin^2\th\la_i^2}\right)\quad \forall 1\leq i\leq n.
\end{equation}
As in \cite{x-y}\cite{j-x-y3}, the slope function of $C_{f,\th}$ is thereby
\begin{equation}
W=\prod_{j=0}^n \sec\th_j=\sec\th \prod_{j=1}^n \f{\sqrt{\cos^2\th +\sin^2\th \la_j^2}}{\cos\th}.
\end{equation}

 \bigskip

\subsection{Proof of Proposition \ref{case1}}\label{App3}

Let $D$ be the bounded closed domain on the $\varphi\psi$-plane enclosed by the line segment from $(0,0)$ to $(\varphi_0,0)$
and the graph of function $h:[0,\varphi_0]\ra \R$ given by
\begin{equation}
h(\varphi)=\f{\big(\f{(\la^2-1)p}{1+\la^2\varphi^2}-(n-p)\big)\varphi}{c(n-p)},
\end{equation}
where $c\in (0,1]$ is a constant to be chosen.
$$\includegraphics[scale=0.8]{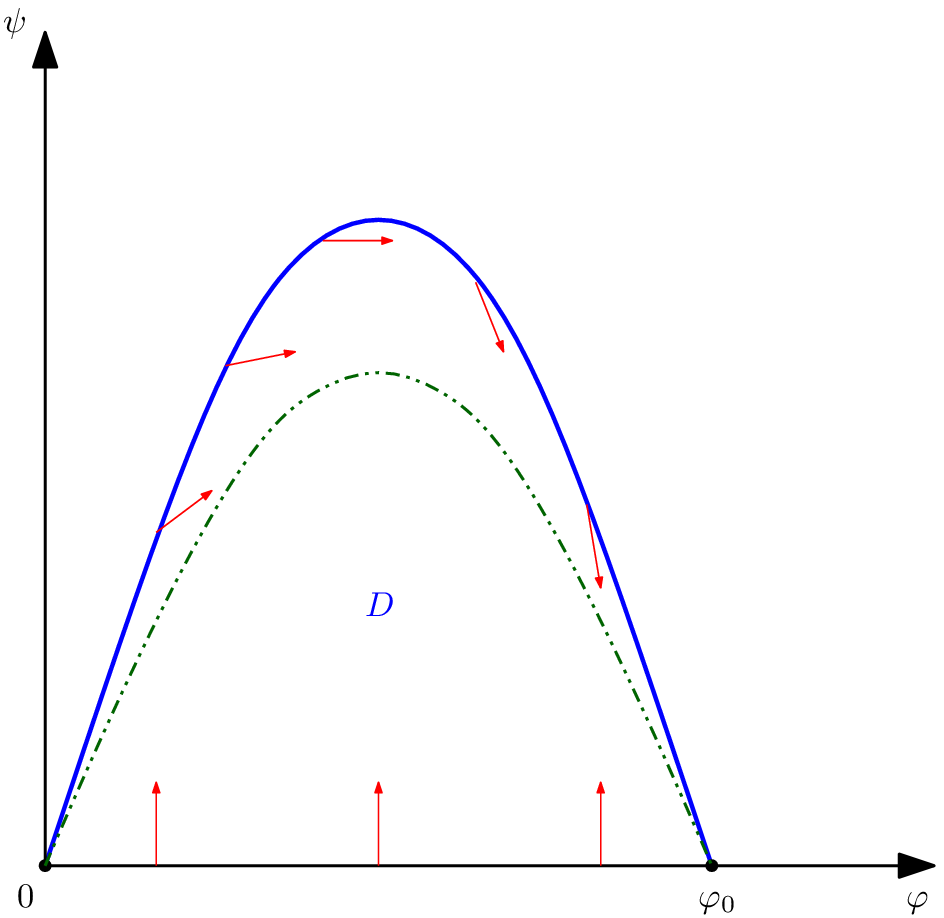}$$

 We shall prove that
$D$ is invariant under the forward development of (\ref{ODE2}) by verifying that $X=(X_1,X_2)$ points inward in $\p D$
except at the zero points $(0,0)$ and $(\varphi_0,0)$. In other words, we need to show:
\begin{enumerate}
\item[(A)] $X_2(\varphi,0)>0$ for $\varphi\in (0,\varphi_0)$;
\item[(B)] $h'(\varphi)> \f{X_2}{X_1}(\varphi,h(\varphi))$ for $\varphi\in (0,\varphi_0)$.
\end{enumerate}

Here (A) is obvious and (B) requires following careful calculations.


Set
\begin{equation}\label{f12}
\aligned
f_1(\varphi)&:=\f{(\la^2-1)p}{1+\la^2\varphi^2}-(n-p),\\
f_2(\varphi)&:=n-p+\f{p}{1+\la^2\varphi^2},
\endaligned
\end{equation}
and
\begin{equation}\label{h}
h(\varphi)=\f{ f_1(\varphi)\varphi}{c(n-p)}.
\end{equation}
Then
\begin{equation}\label{dh}
h'(\varphi)=\f{f_1(\varphi)+ f'_1(\varphi)\varphi}{c(n-p)}.
\end{equation}
Due to (\ref{X12}), (\ref{f12}), (\ref{h}) and (\ref{dh}), (B) is equivalent to
\begin{equation}\label{neq}\aligned
0&< h'(\varphi)+1+ \Big[n-p+\f{p}{1+\la^2\varphi^2}+\big(n-p+\f{(1-\la^2)p}{1+\la^2\varphi^2}\big)\f{\varphi}{h(\varphi)}\Big]\big[1+(\varphi+h(\varphi))^2\big]\\
&=\left(1+\f{f_1(\varphi)}{c(n-p)}\right)+\f{ f'_1(\varphi)\varphi}{c(n-p)}+(f_2(\varphi)-c(n-p))\left[1+\varphi^2\big(1+\f{f_1(\varphi)}{c(n-p)}\big)^2\right]\\
:&=I+II+III\cdot IV.
\endaligned
\end{equation}

By (\ref{varphi0}),
\begin{equation}
\la^2\varphi_0^2=\f{\la^2 p-n}{n-p},\qquad 1+\la^2\varphi_0^2=\f{(\la^2-1)p}{n-p}.
\end{equation}
Set
\begin{equation}
s:=\f{1+\la^2\varphi_0^2}{1+\la^2\varphi^2}-1.
\end{equation}
Then $\varphi\in (0,\varphi_0)$ implies $s\in (0,\la^2\varphi_0^2)=(0,\f{\la^2 p-n}{n-p})$, and
\begin{equation}\aligned
\f{1}{1+\la^2\varphi^2}&=\f{1+s}{1+\la^2\varphi_0^2}=\f{n-p}{(\la^2-1)p}(1+s),\\
\la^2\varphi^2&=\f{1+\la^2\varphi_0^2}{1+s}-1=\f{\la^2\varphi_0^2-s}{1+s}=\f{\f{\la^2 p-n}{n-p}-s}{1+s}.
\endaligned
\end{equation}
It immediately follows that
\begin{equation}
f_1(\varphi)=(n-p)s,
\end{equation}
\begin{equation}
f_2(\varphi)=\f{(n-p)(\la^2+s)}{\la^2-1}
\end{equation}
and
\begin{equation}\aligned
 f'_1(\varphi)\varphi&=-\f{2(\la^2-1)p\la^2\varphi^2}{(1+\la^2\varphi^2)^2}\\
&=-2(\la^2-1)p\big(1-\f{1}{1+\la^2\varphi^2}\big)\f{1}{1+\la^2\varphi^2}\\
&=-\f{2(n-p)}{(\la^2-1)p}(\la^2 p-n-(n-p)s)(1+s).
\endaligned
\end{equation}
Therefore
\begin{equation}
I=1+\f{s}{c}:=I(s),
\end{equation}
\begin{equation}
II=-\f{2(n-p)}{c(\la^2-1)p}\left(\f{\la^2 p-n}{n-p}-s\right)(1+s):=II(s),
\end{equation}
\begin{equation}
III=\f{n-p}{\la^2-1}(\la^2-c(\la^2-1)+s):=III(s)
\end{equation}
and
\begin{equation}\aligned
IV&=1+\varphi^2(1+\f{f_1(\varphi)}{c(n-p)})^2\\
&=1+\la^{-2}(\la^2\varphi^2)(1+\f{f_1(\varphi)}{c(n-p)})(1+\f{f_1(\varphi)}{c(n-p)})\\
&=1+\la^{-2}(\f{\la^2 p-n}{n-p}-s)\f{1+\f{s}{c}}{1+s}(1+\f{s}{c})\\
&\geq 1+\la^{-2}(\f{\la^2 p-n}{n-p}-s)(1+\f{s}{c})\\
:&=IV(s).
\endaligned
\end{equation}
Let
\begin{equation}
F(s):=I(s)+II(s)+III(s)\cdot IV(s).
\end{equation}
By $c\in(0,1]$ and $s>0$, $III(s)>0$.
So $I+II+III\cdot IV\geq F(s)$ for $s\in (0,\la^2\varphi_0^2)$, i.e., $\varphi\in (0,\varphi_0)$.
Observe that $F(s)$
is a cubic polynomial in $s$ and the coefficient of the third order term is $-\f{n-p}{c\la^2(\la^2-1)}<0$. Hence $F(s)=F(0)+sG(s)$, where $G(s)$ is a quadratic polynomial whose graph is a parabola opening downward.
 This implies $G(s)\geq \min\{G(0),G(\la^2\varphi_0^2)\}$ for $s\in (0,\la^2\varphi_0^2)$. Therefore, for (\ref{neq}), it suffices to show
\begin{itemize}
\item $F(0)\geq 0$;
\item $G(0)> 0$;
\item $G(\la^2\varphi_0^2)> 0$.
\end{itemize}

A straightforward calculation shows
\begin{equation}
\aligned
F(0)&=I(0)+II(0)+III(0)\cdot IV(0)\\
&=1-\f{2(\la^2 p-n)}{c(\la^2-1)p}+(n-p)(\f{\la^2}{\la^2-1}-c)(1+\f{\la^2p-n}{\la^2(n-p)})\\
&=1+n-\f{2(\la^2p-n)}{c(\la^2-1)p}-\f{c(\la^2-1)n}{\la^2},
\endaligned
\end{equation}
\begin{equation}\label{G0}
\aligned
G(0)=&F'(0)=I'(0)+II'(0)+III'(0)\cdot IV(0)+III(0)\cdot IV'(0)\\
=&\f{1}{c}-\f{2(n-p)}{c(\la^2-1)p}(\f{\la^2 p-n}{n-p}-1)+\f{n-p}{\la^2-1}(1+\f{\la^2 p-n}{\la^2 (n-p)})\\
&+(n-p)(\f{\la^2}{\la^2-1}-c)\la^{-2}(\f{\la^2 p-n}{c(n-p)}-1)\\
=&(\f{1}{c}-1)p-\f{1}{c}+\f{(\f{4-p}{c}-p)(n-p)}{\la^2p -p}+\f{(c+2)n-cp}{\la^2},\\
\endaligned
\end{equation}
\begin{equation}
\aligned
F(\la^2\varphi_0^2)&=I(\la^2\varphi_0^2)+II(\la^2\varphi_0^2)+III(\la^2\varphi_0^2)\cdot IV(\la^2\varphi_0^2)\\
&=1+\f{\la^2 p-n}{c(n-p)}+\f{n-p}{\la^2-1}(\la^2-c(\la^2-1)+\f{\la^2p-n}{n-p})\\
&=1+n+\f{\la^2 p-n}{c(n-p)}-c(n-p)
\endaligned
\end{equation}
and
\begin{equation}
\aligned
G(\la^2\varphi_0^2)&=\f{F(\la^2\varphi_0^2)-F(0)}{\la^2\varphi_0^2}\\
&=\f{1}{c}+\f{c(n-p)}{\la^2}+\f{2(n-p)}{c(\la^2-1)p}>0.
\endaligned
\end{equation}

Recalling $\la^2=\frac{k(k+n-1)}{p}$, we choose $c$ according to the values of $(n,p,k)$:

\textbf{Case 1.} $(n,p,k)=(3,2,2)$.

Using $c=1$, we have
$$F(0)=4-\f{5}{3c}-\f{9c}{4}=\f{1}{12}>0$$
and
$$G(0)=\f{4}{3c}-\f{5}{6}+\f{c}{4}=\f{3}{4}>0.$$

\textbf{Case 2.} $(n,p,k)=(5,4,2)$.

With $c=1$,
$$F(0)=6-\f{7}{4c}-\f{10c}{3}=\f{11}{12}>0$$
and
$$G(0)=\f{3}{c}-\f{7}{6}+\f{c}{3}=\f{13}{6}>0.$$

\textbf{Case 3.} $(n,p,k)=(5,4,4)$.

For $c=\f{6}{7}$,
$$F(0)=6-\f{27}{14c}-\f{35c}{8}=0$$
and
$$G(0)=\f{3}{c}-\f{81}{28}+\f{c}{8}=\f{5}{7}>0.$$

\textbf{Case 4.} $n\geq 7$.

In this case, Theorem \ref{npk2} asserts $p<n<2p$ and $p\geq 4$.

Take $c=\f{1}{2}$. By $n>p$,
$$\aligned
F(0)&=1+n-\f{2(\la^2p-n)}{c(\la^2-1)p}-\f{c(\la^2-1)n}{\la^2}\\
&\geq 1+n-\f{2}{c}-cn=\f{n}{2}-3>0.
\endaligned$$
From $2p>n$ we have
\begin{equation*}\aligned
&(\la^2 p-p)-3(n-p)=k(k+n-1)-p-3(n-p)\\
\geq & 2(n+1)-p-3(n-p)=2p-n+2>0,
\endaligned
\end{equation*}
i.e., $\f{n-p}{\la^2 p-p}<\f{1}{3}$. Hence
\begin{equation*}\aligned
G(0)&=p-2+\f{(8-3p)(n-p)}{\la^2 p-p}+\f{5n-p}{2\la^2}\\
&>p-2+\f{8-3p}{3}>0
\endaligned
\end{equation*}

Therefore we establish (B) that $D$ is invariant under the forward development of (\ref{ODE2}).
Since $(0,0)$ is a saddle critical point, there exists a smooth solution $t\in (-\infty,T_\infty)\mapsto (\varphi(t),\psi(t))\in \R^2$ to (\ref{ODE2}), with
$\lim\limits_{t\ra -\infty}(\varphi(t),\psi(t))=(0,0)$. Here
$T_\infty\in \R\cup \{+\infty\}$ such that $(-\infty,T_\infty)$ is the maximal existence interval of this solution. Moreover, by Theorem 3.5 in \S VIII of \cite{h},
as $t\ra -\infty$,
$\varphi(t)=O(e^{\mu_1 t})$, $\psi(t)=O(e^{\mu_1 t})$ and the direction of $(\varphi(t),\psi(t))^T$
converges to that of $V_1$, i.e., an eigenvector of $A$ associated to $\mu_1$ (see (\ref{A}), (\ref{la12}) and (\ref{V12})).
 It is easy to check that $h'(0)>\mu_1$.
 Thus the orbit of this solution remains in $D$ and $T_\infty=+\infty$.
 By (A), we know $\varphi'(t)=\psi(t)>0$.
 Hence the $\om$-limit set of the orbit must be a critical point, not a limit cycle, as $t$ tends to positive infinity.
Now we complete the proof.
\bigskip

\subsection{Proof of Proposition \ref{case2}}\label{App4}

The proof relies on the following lemma.

\begin{lem}\label{lem}
For $(n,p)=(3,2)$, $k\geq 4$ or $(n,p)=(5,4)$, $k\geq 6$,
let $t\in [b_0,b_2]\mapsto (\varphi(t),\psi(t))$ be a smooth solution to (\ref{ODE2}) and  $b_1\in (b_0,b_2)$ so that
\begin{itemize}
\item $\varphi(b_0)\geq \sqrt{\f{3p-n-1}{3(n-p)}}$;
\item $\psi(b_0)=\psi(b_1)=\psi(b_2)=0$;
\item $\psi(t)>0$ for $t\in (b_0,b_1)$, and $\psi(t)<0$ for $t\in (b_1,b_2)$.
\end{itemize}
Then $\varphi(b_1)>\varphi_0$ and $\varphi(b_0)<\varphi(b_2)<\varphi_0$.
Namely, there are no limit cycles of (\ref{ODE2}) on the region $\varphi\geq\sqrt{\f{3p-n-1}{3(n-p)}}$.
\end{lem}


\begin{proof}
Using symbols in Appendix \ref{App3},
we have from (\ref{ODE2}) that
\begin{equation}\label{psit}
\psi_t=-\psi-(f_2(\varphi)\psi-f_1(\varphi)\varphi)\big[1+(\varphi+\psi)^2\big].
\end{equation}
By assumptions, $\varphi(b_1)\neq \varphi_0$ and
$0\geq \psi'(b_1)=f_1(\varphi(b_1))\varphi(b_1)\big(1+\varphi(b_1)^2\big)$.
So $\varphi(b_1)>\varphi_0$ and $\psi'(b_1)<0$.
Similarly $\varphi(b_0),\varphi(b_2)<\varphi_0$.

For $t\in(b_1,b_2)$, with
\begin{equation}\label{td}
\left\{\begin{array}{ll}
\td{\varphi}=\varphi\\
\td{\psi}=-\psi
\end{array}\right.
\end{equation}
(\ref{ODE2}) becomes
\begin{equation}\label{forY}
\left\{\begin{array}{ll}
\td{\varphi}_t=-\td{\psi}\\
\td{\psi}_t=-\td{\psi}-\big(f_2(\td{\varphi})\td{\psi}+f_1(\td{\varphi})\td{\varphi}\big)\big[1+(\td{\varphi}-\td{\psi})^2\big]
\end{array}\right.
\end{equation}
By the monotonicity, $\psi$ for $t\in(b_0,b_1)$ and $\td{\psi}$ for $t\in(b_1,b_2)$
can be written as smooth functions $\psi(\varphi)$ and $\td{\psi}(\varphi)$ respectively.
Then
we have
\begin{equation}
\dfrac{d\psi}{d\varphi}=-1-\left[f_2(\varphi)-f_1(\varphi)\dfrac{\varphi}{\psi}\right]\left[1+\left(\varphi+\psi\right)^2\right],
\end{equation}
and
\begin{equation}
\dfrac{d\td{\psi}}{d\varphi}=\,\,1+\left[f_2(\varphi)+f_1(\varphi)\dfrac{\varphi}{\td{\psi}}\right]\left[1+\left(\varphi-\td{\psi}\right)^2\right].
\end{equation}
Therefore
\begin{equation}
\dfrac{d\psi}{d\varphi}-\dfrac{d\td{\psi}}{d\varphi}<
f_1(\varphi)
\left\{
\dfrac{\varphi}{\psi}\left[1+\left(\varphi+\psi\right)^2\right]-
\dfrac{\varphi}{\td{\psi}}\left[1+\left(\varphi-\td{\psi}\right)^2\right]
\right\}.
\end{equation}
Note that $\psi-\td{\psi}$ is continuous on $[\varphi_0,\varphi(b_1)]$ with value zero at $\varphi(b_1)$.
Through a contradiction argument, we have $\psi>\td{\psi}$ on $[\varphi_0,\varphi(b_1))$.
$$\includegraphics[scale=0.8]{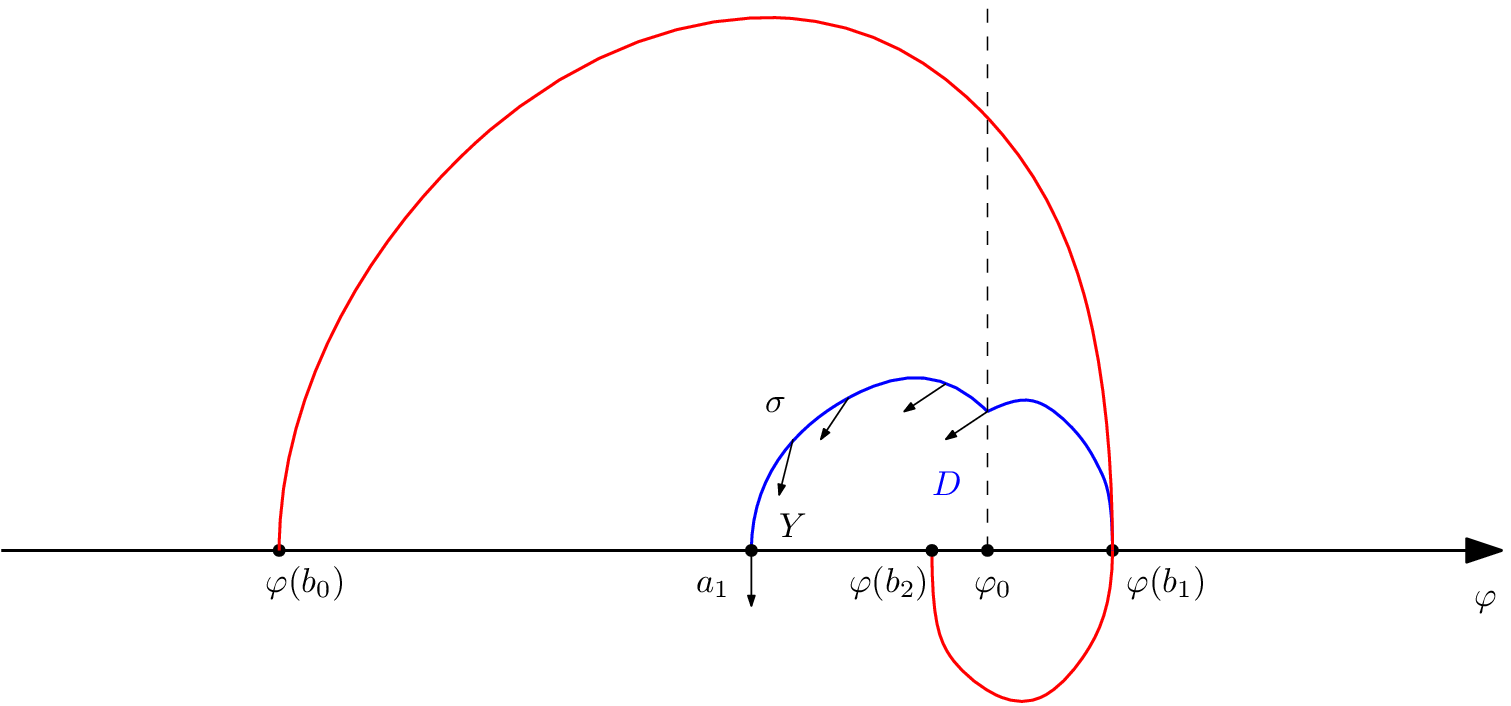}$$

Let $\sigma$ be the orbit of the backward solution to \eqref{ODE2} from $(\varphi_0, \td{\psi}(\varphi_0))$ to $(a_1,0)$ for some $a_1>\varphi(b_0)$.
Based on \eqref{forY}, set
\begin{equation}\label{Y12}
\aligned
Y_1(\varphi,\psi)&=-\psi,\\
Y_2(\varphi,\psi)&=-\psi-(f_2(\varphi)\psi+f_1(\varphi)\varphi)\big[1+(\varphi-\psi)^2\big].
\endaligned
\end{equation}
If we have
\begin{itemize}
\item [$(\star):$]
$\left|\begin{array}{cc} X_1 & X_2 \\ Y_1 & Y_2\end{array}\right|<0$ at $(\varphi,\psi)$ when $\varphi\geq \sqrt{\f{3p-n-1}{3(n-p)}}$
and $\psi>0$,
\end{itemize}
then the inequality holds
on $\sigma$.
Hence the region $D$ embraced by $\sigma$, the $\varphi$-axis and the striaight line $\varphi=\varphi_0$
forms an invariant set under the forward development of \eqref{forY}.
Therefore $\varphi(b_0)<a_1<\varphi(b_2)<\varphi_0$.

Since $X_1=-Y_1=\psi>0$,
$(\star)$ is equivalent to show
$Y_2+X_2<0$.
So we do the following calculations. 
$$\aligned
Y_2+X_2=&-\psi-(f_2(\varphi)\psi+f_1(\varphi)\varphi)\big[(1+(\varphi-\psi)^2\big]\\
&-\psi-(f_2(\varphi)\psi-f_1(\varphi)\varphi)\big[1+(\varphi+\psi)^2\big]\\
=&-2\psi-2f_2(\varphi)\psi(1+\varphi^2+\psi^2)+4f_1(\varphi)\varphi^2\psi\\
\leq & 2\psi(-1-f_2(\varphi)(1+\varphi^2)+2f_1(\varphi)\varphi^2)\\
=&\f{2\psi}{1+\la^2\varphi^2}\Big[-1-\la^2\varphi^2-(1+\varphi^2)\big((n-p)(1+\la^2\varphi^2)+p\big)\\
&\qquad\quad+2\varphi^2\big((\la^2-1)p-(n-p)(1+\la^2\varphi^2)\big)\Big]\\
=&\f{2\psi}{1+\la^2\varphi^2}\left[-3\la^2(n-p)\varphi^4+\big(\la^2(3p-n-1)-3n\big)\varphi^2-n-1\right]\\
<&-\f{6\la^2(n-p)\psi}{1+\la^2\varphi^2}\varphi^2(\varphi^2-\f{3p-n-1}{3(n-p)})\leq 0.
\endaligned$$
Now the proof of the lemma gets complete.
\end{proof}



As in Appendix \ref{App3}, there exists a smooth solution $t\in(-\infty,T_\infty)\mapsto (\varphi(t),\psi(t))$ to (\ref{ODE2}), with $\lim\limits_{t\ra -\infty}(\varphi(t),\psi(t))=(0,0)$, $\varphi(t)=O(e^{\mu_1 t})$,
$\psi(t)=O(e^{\mu_1 t})$ and the direction of $(\varphi(t),\psi(t))^T$ convergent to that of $V_1$ as $t\ra -\infty$.
We shall accomplish the proof of Proposition
\ref{case2} in several steps.

\textbf{Step 1. }
{Show the existence of $t_1\in (-\infty,T_\infty)\subset \R$, such that $\psi(t)>0$ for all $t\in (-\infty,t_1]$ with}
\begin{equation}\label{step1}
\varphi(t_1)=\varphi_0\quad \text{and }\quad \psi(t_1)\leq \f{1}{5}\varphi_0.
\end{equation}

Define $g:[0,\varphi_0]\ra \R$ by
\begin{equation}
g(\varphi)=(2f_1(\varphi)+\f{1}{5})\varphi.
\end{equation}
Let $D$ be the domain enclosed by
the graph of $g$, the $\varphi$-axis and the line $\varphi=\varphi_0$.

$$\includegraphics[scale=0.8]{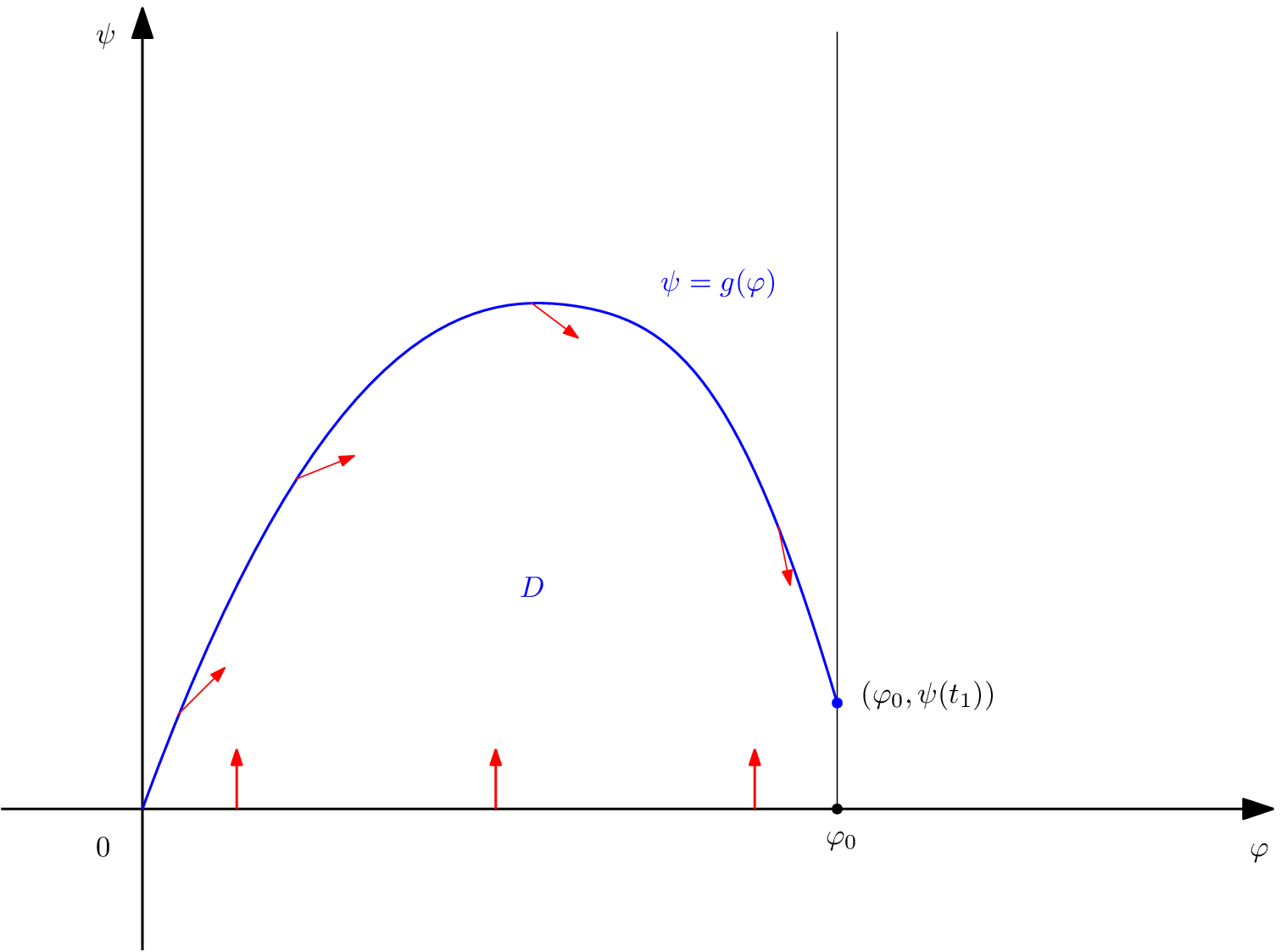}$$
We claim that  the vector field $X$ points inward on $\partial D-\{\varphi=\varphi_0\}$.
Namely,
\begin{enumerate}
\item[(A)] $X_2(\varphi,0)>0$ for each $\varphi\in (0,\varphi_0)$;
\item[(B)] $g'(\varphi)> \f{X_2}{X_1}(\varphi,g(\varphi))$ for any $\varphi\in (0,\varphi_0)$.
\end{enumerate}
Here (A) is trivial and (B) is equivalent to
\begin{equation}
\aligned
0<& g'(\varphi)+1+\left(f_2(\varphi)-\f{\varphi f_1(\varphi)}{g(\varphi)}\right)\big[1+(\varphi+g(\varphi))^2\big]\\
=&\left(\f{6}{5}+2f_1(\varphi)\right)-\Big(-2 f'_1(\varphi)\varphi\Big)+\left(f_2(\varphi)-\f{f_1(\varphi)}{2f_1(\varphi)+\f{1}{5}}\right)\left[1+\varphi^2(\f{6}{5}+2f_1(\varphi))^2\right]\\
:=&I-II+III\cdot IV.
\endaligned
\end{equation}
As in \S \ref{App3}, we use
$$s:=\f{1+\la^2\varphi_0^2}{1+\la^2\varphi^2}-1\quad \text{and} \quad s\in (0,\la^2\varphi_0^2).$$
Similarly, we have (now $n-p=1$ in our cases)
\begin{equation}
\aligned
I&=\f{6}{5}+2s,\\
II&=\f{4}{(\la^2-1)p}(\la^2p-n-s)(1+s),\\
III&=\f{\la^2+s}{\la^2-1}-\f{s}{2s+\f{1}{5}},\\
IV&=1+\f{\la^2p-n-s}{\la^2(1+s)}(\f{6}{5}+2s)^2=1+\f{(\la^2-1)p}{4\la^2}\left(\f{\f{6}{5}+2s}{1+s}\right)^2\cdot II.
\endaligned
\end{equation}
Therefore
\begin{equation}\label{es1}\aligned
&I-II+III\cdot IV\\
=&I+III+\left[\f{(\la^2-1)p}{4\la^2}\left(\f{\f{6}{5}+2s}{1+s}\right)^2\cdot III-1\right]\cdot II\\
\geq &I+III+\left[\f{p}{4}\left(\f{\f{6}{5}+2s}{1+s}\right)^2\left(1-\f{s}{2s+\f{1}{5}}\right)-1\right]\cdot II
\endaligned
\end{equation}
Set
\begin{equation}
F(s):=\left(\f{\f{6}{5}+2s}{1+s}\right)^2\left(1-\f{s}{2s+\f{1}{5}}\right)=\f{4}{25}\left(\f{3+5s}{1+s}\right)^2\f{1+5s}{1+10s}.
\end{equation}
Then
$$\log F=\log \f{4}{25}+2\log(3+5s)-2\log(1+s)+\log(1+5s)-\log(1+10s)$$
and
$$\aligned
\f{d\log F}{ds}&=\f{10}{3+5s}-\f{2}{1+s}+\f{5}{1+5s}-\f{10}{1+10s}\\
&=\f{-11+20s+175s^2}{(3+5s)(1+s)(1+5s)(1+10s)}.
\endaligned$$
 Hence $F'(s)=0(>0,<0)$ if and only if $s=\f{1}{5}(>\f{1}{5},<\f{1}{5})$, and
\begin{equation}\label{F}
\min_{s\in (0,\infty)}F=F(\f{1}{5})=\f{32}{27}.
\end{equation}

For $(n,p)=(5,4)$, substituting \eqref{F} into \eqref{es1} leads to
\begin{equation}
I-II+III\cdot IV\geq I+III+\f{5}{27}II> 0.
\end{equation}

For $(n,p)=(3,2)$, it then produces
\begin{equation}
\aligned
&I-II+III\cdot IV\geq I+III-\f{11}{27}II\\
=&\f{6}{5}+2s+\f{\la^2+s}{\la^2-1}-\f{s}{2s+\f{1}{5}}-\f{22}{27(\la^2-1)}(2\la^2-3-s)(1+s)\\
\geq& \f{6}{5}+2s+1-\f{s}{2s+\f{1}{5}}-\f{44}{27}(1+s)\\
=&\f{19}{270}+\f{10}{27}s+\f{1}{20s+2}> 0.
\endaligned
\end{equation}

Hence (B) holds for both cases.

Since $g'(0)>\mu_1$, the solution develops in $D$ until it hits
the border line $\varphi=\varphi_0$ at $t_1\in \R$
or it approaches $(\varphi_0,0)$ as $t\rightarrow +\infty$.
Due to the fact that $(\varphi_0,0)$ is a spiral point,
the latter cannot occur and moreover $t_1<+\infty$, $\varphi(t_1)=\varphi_0$ and $\psi(t_1)>0$.

\textbf{Step 2.} 
Before $\psi(t)$ reaches zero, we have $\varphi_t=\psi>0$, $\varphi>\varphi_0$ (after $t_1$)
and $f_1(\varphi)> 0$.
Consequently,
\begin{equation}
(\varphi+\psi)_t=-(f_2(\varphi)\psi-f_1(\varphi)\varphi)\big[1+(\varphi+\psi)^2\big]\leq 0.
\end{equation}
Hence the solution intersects the $\varphi$-axis for the first time when $t$ equals some $T_1\in \R$,
with $\varphi_0<\varphi_1:=\varphi(T_1)\leq \varphi_0+\psi(t_1)\leq \frac{6}{5}\varphi_0$.

\textbf{Step 3.}
At $t=T_1$, $\psi_t<0$.
So the solution dips into the lower half plane and similarly cannot limits to $(\varphi_0,0)$.
By the argument in the proof of Lemma \ref{lem}, the solution extends forward to touch the $\varphi$-axis again (after $T_1$) when $t$ equals some $T_2\in \R$.
Mark $t_2\in(T_1,T_2)$ for $\varphi(t_2)=\varphi_0$.
 When $\psi<0$ and $\varphi\leq\varphi_0$, we have $(\varphi+\psi)_t\geq 0$.
 Therefore, $\varphi_0>\varphi_2:=\varphi(T_2)>\varphi_0+\psi(t_2)\geq \varphi_0-\psi(t_1)\geq \frac{4}{5}\varphi_0$.

\textbf{Step 4.}
By induction,
we obtain
$\{T_i:i\in \Bbb{Z}^+\}$
and $\varphi_i:=\varphi(T_i)$
with properties:
\begin{itemize}
\item $\psi(T_i)=0$ for each $i\in \Bbb{Z}^+$;
\item $\{\varphi_{2m-1}:m\in \Bbb{Z}^+\}$ is a strictly decreasing sequence in $(\varphi_0,\frac{6}{5}\varphi_0]$,\\ $\{\varphi_{2m}:m\in \Bbb{Z}^+\}$
is a strictly increasing sequence in $[\f{4}{5}\varphi_0,\varphi_0)$;
\item $\psi(t)>0$ in $(-\infty,T_1)\cup\left(\bigcup\limits_{m\in \Bbb{Z}^+}(T_{2m},T_{2m+1})\right)$;
\item $\psi(t)<0$ in $\bigcup\limits_{m\in \Bbb{Z}^+}(T_{2m-1},T_{2m})$.
\end{itemize}

\textbf{Step 5.}
Assume $a:=\lim\limits_{m\ra \infty}\varphi(T_{2m})<\varphi_0$.
Then there would be a limit cycle for \eqref{ODE2} through $(a,0)$.
But $a>\f{4}{5}\varphi_0>\sqrt{\f{3p-n-1}{3(n-p)}}$.
It leads to a contradiction to the nonexistence of limit cycles in Lemma \ref{lem}.
The same for $\lim\limits_{m\ra \infty}\varphi(T_{2m-1})$.
Therefore, $\lim\limits_{m\ra \infty}\varphi(T_{2m})=\lim\limits_{m\ra \infty}\varphi(T_{2m-1})=\varphi_0$.

Since the solution cannot attain $(\varphi_0,0)$ in a finite time,
it is now clear that $T_i\rightarrow +\infty$.
This
completes the proof of Proposition \ref{case2}.


\section*{Acknowledgements} Research supported in part by the
NSFC (Grant Nos. 11471299, 11471078, 11622103, 11526048, 11601071, 11871445),
the Fundamental Research Funds for the Central Universities, the SRF for ROCS, SEM,
and a Start-up Research Fund from Tongji University. It is a great pleasure to thank
the referees for helpful comments, and MSRI,  Chern Institute at Nankai University,
ICTP and IHES for hospitalities.

\bibliographystyle{amsplain}

\end{document}